\documentclass[12pt]{amsart}
\usepackage{amssymb,amsxtra,amsmath,amsfonts}
\usepackage{verbatim}
\usepackage{mathtools}
\usepackage{bbm} %\mathbbm{1}
\newtheorem{theorem}{Theorem}
\newtheorem{lemma}[theorem]{Lemma}
\newtheorem{proposition}[theorem]{Proposition}
\newtheorem{corollary}[theorem]{Corollary}

\theoremstyle{definition}
\newtheorem{definition}[theorem]{Definition}
\newtheorem{example}[theorem]{Example}

\theoremstyle{remark}
\newtheorem{remark}[theorem]{Remark}

\numberwithin{equation}{section}
\numberwithin{theorem}{section}

\newcommand\thref{Theorem \ref}
\newcommand\leref{Lemma \ref}
\newcommand\prref{Proposition \ref}

\newcommand\deref{Definition \ref}
\newcommand\exref{Example \ref}

\newcommand\seref{Section \ref}
\renewcommand{\comment}[1]{}
\def\ii{\mathrm{i}} %{\sqrt{-1}}
\def\vac{\boldsymbol{1}}
\def\CC{\mathbb{C}}
\def\RR{\mathbb{R}}
\def\ZZ{\mathbb{Z}}
\def\NN{\mathbb{N}}
\def\N{\mathcal{N}}

\def\A{\mathcal{A}}
\def\ph{\varphi}
\def\om{\omega}
\def\si{\sigma}
\def\lieh{\mathfrak{h}}
\def\hh{\hat{\lieh}}
\def\hhp{\hh_\ph}

\def\al{\alpha}
\def\be{\beta}
\def\th{\theta}
\def\ga{\gamma}
\def\de{\delta}
\def\la{\lambda}
\def\ep{\varepsilon}
\def\ze{\zeta}

\def\Id{\mathrm{Id}}
\def\P{\mathcal{P}}
\def\La{\Lambda}
\newcommand{\lie}[1]{\mathfrak{#1}}
\newcommand{\h}{\mathfrak{h}}
\newcommand{\nop}[1]{\textnormal{{:}}#1\textnormal{{:}}}
\newcommand{\nopm}[1]{{}^\circ_\circ \hspace{-.7sp}#1 \hspace{-.5sp} {}^\circ_\circ}

\DeclareMathOperator\re{Re}
\DeclareMathOperator\im{Im}

\DeclareMathOperator\LF{LFie}
\DeclareMathOperator\BLF{\overline{LFie}}

\DeclareMathOperator\End{End}
\DeclareMathOperator\Span{span}
\DeclareMathOperator\ind{Ind}
\DeclareMathOperator\aut{Aut}

\DeclareMathOperator\der{Der}
\DeclareMathOperator\res{Res}

\DeclareMathOperator\sgn{sgn}
\DeclareMathOperator\Li{\mathrm{Li}}

\begin{document}

\title{Twisted logarithmic modules of lattice vertex algebras}
\author{Bojko Bakalov}
\address{Department of Mathematics, North Carolina State University, 
Raleigh, NC 27695, USA}
\email{bojko\_bakalov@ncsu.edu}

\author{McKay Sullivan}
\address{Department of Mathematics, Dixie State University, Saint George, UT 84770, USA}
\email{mckay.sullivan@dixie.edu}

\begin{abstract}
Twisted modules over vertex algebras formalize the relations among twisted vertex operators and have applications to conformal field theory and representation theory. A recent generalization, called twisted logarithmic module, involves the logarithm of the formal variable and is related to logarithmic conformal field theory. We investigate twisted logarithmic modules of lattice vertex algebras, reducing their classification to the classification of modules over a certain group. This group is a semidirect product of a discrete Heisenberg group and a central extension of the additive group of the lattice.
\end{abstract}

%\thanks{The first author is supported in part by a Simons Foundation grant 279074} 

\date{August 28, 2017}

%\keywords{Lattice vertex algebra; twisted module; vertex algebra}


\maketitle

\tableofcontents

%\nocite{*}

%%%%%%%%%%%%%%%%%%%%%%%%%%%%%%%%%%%%%%%%%%%%%%%%%
%%%%%%%%%%%%%%%%%%%%%%%%%%%%%%%%%%%%%%%%%%%%%%%%%
\section{Introduction}
%%%%%%%%%%%%%%%%%%%%%%%%%%%%%%%%%%%%%%%%%%%%%%%%%
%%%%%%%%%%%%%%%%%%%%%%%%%%%%%%%%%%%%%%%%%%%%%%%%%

The notion of a \textit{vertex algebra} was introduced by Borcherds \cite{Bo} and developed in \cite{FLM2, K2, FB, LL, KRR} among many other works. 
It provides a rigorous algebraic formulation of two-dimensional chiral conformal field theory \cite{BPZ, Go, DMS}, 
and is a powerful tool in the representation theory of infinite-dimensional Lie algebras \cite{K1,KRR}. 
The main motivating examples of vertex algebras were the \emph{lattice vertex algebras} \cite{Bo},
which generalized the Frenkel--Kac realization of affine Kac--Moody algebras in terms of vertex operators \cite{FK}. 
In many applications vertex operators appeared in a ``twisted'' form, as in the principal realization of affine Kac--Moody algebras \cite{LW,KKLW} and in the
Frenkel--Lepowsky--Meurman construction of a vertex algebra with a natural action of the Monster group on it \cite{FLM2,Ga}.
This led to the notion of a \emph{twisted module} of a vertex algebra \cite{D,FFR} that axiomatizes the properties of twisted vertex operators \cite{KP,Le,FLM1}.
One of the important applications of twisted modules is to orbifolds in conformal field theory (see e.g.\ \cite{DVVV,KT,DLM,BE}). 
Twisted modules over lattice vertex algebras were classified in \cite{D} in a special case and in the general case in \cite{BK}.

The notion of a $\ph$-twisted module was defined in \cite{D,FFR} only for finite-order automorphisms $\ph$.
In particular, both $\ph$ and the Virasoro $L_0$ operator were assumed to be semisimple. The idea of deforming $L_0$ with an automorphism of infinite order first appeared in \cite{AM2}.
Shortly thereafter, the notion of a $\varphi$-twisted module was generalized by Huang to the case when $\varphi$ is not semisimple  \cite{H}, and was
further developed in \cite{Bak,BS,HY,Y}.
This generalization was motivated by logarithmic conformal field theory (see e.g.\ \cite{Kau1,AM,CR}), and has promising applications to Gromov--Witten theory (cf.\ \cite{DZ,M,MT,BM,BW}).
The main feature of such modules is that the twisted fields involve the logarithm of the formal variable; for this reason we also call them \emph{twisted logarithmic modules}.

In \cite{Bak} the first author gave a different definition of twisted logarithmic modules, which allowed him to prove a Borcherds identity for these modules. Twisted logarithmic modules according to Huang's original definition are also twisted logarithmic modules according to the definition in \cite{Bak}, and under certain additional assumptions, the definition in \cite{Bak} implies Huang's definition  (see \cite{HY} for details).

In this paper we initiate the study of $\ph$-twisted modules over a lattice vertex algebra $V_Q$, where $Q$ is an integral lattice and $\ph$ is an arbitrary automorphism of $Q$ (which is then lifted to an automorphism of $V_Q$). The first step is to restrict such a module $M$ to the Heisenberg subalgebra $B^1(\lieh)\subset V_Q$, where $\lieh=\CC\otimes_\ZZ Q$ is a vector space with the bilinear form induced from $Q$.
Then $\ph$ is also an automorphism of $B^1(\lieh)$, and $M$ is a $\ph$-twisted $B^1(\lieh)$-module. Such modules were studied in our previous works \cite{Bak,BS}. 
%There we also determined the action of the Virasoro algebra on $M$, which will be useful here.

The next step is to find how the action of $B^1(\lieh)$ on $M$ extends to $V_Q$, i.e., to find the twisted vertex operators $Y(e^\la,z)$ for $\la\in Q$. Similarly to \cite{KP,Le,FLM1,D,BK}, these vertex operators are expressed as certain exponentials in the Heisenberg generators times some operators $U_\la$ commuting with the nonzero modes of the Heisenberg algebra. Thus, the classification of $\ph$-twisted $V_Q$-modules is reduced to that of modules over the $\ph$-twisted Heisenberg algebra that are equipped with a compatible action of a certain group $G_\ph$. This is the main result of the present paper (see \thref{main} below). 
Although we follow the approach of \cite{BK}, our setting here is more involved. For instance, when $\ph$ has a finite order, the group $G_\ph$ contains as a subgroup a torus corresponding to the subspace of $\ph$-invariant vectors in $\lieh$. In the general case, this torus is replaced by a \emph{Heisenberg group}. Another difficulty is that when computing the products of twisted vertex operators $Y(e^\la,z)$, we need to utilize \emph{special functions} such as the Lerch transcendent $\Phi$ and the digamma function $\Psi$ (see e.g.\ \cite{BE1,BE2}).

Here is a more detailed description of the contents of the paper. In \seref{sprelim}, we briefly recall the definition of a vertex (super)algebra and our main object, lattice vertex algebras (see \cite{FLM2, K2, FB, LL, KRR}). Then we review the definition and properties of twisted logarithmic modules of vertex algebras, following \cite{Bak}. 

In \seref{sphtwvo}, we start our investigation of the structure of $\ph$-twisted modules over a lattice vertex algebra $V_Q$, where the automorphism $\ph$ is lifted from an automorphism of the lattice $Q$. We find an expression for the twisted vertex operators $Y(e^\la,z)$ in terms of the twisted Heisenberg algebra and certain operators $U_\la$ (see \thref{twistedvop}).

Next, in \seref{sptvo}, we compute the products of twisted vertex operators $Y(e^\la,z)$, 
which allows us to determine the relations satisfied by the operators $U_\la$ (see \thref{commutant}).

Motivated by these relations, in \seref{srtgt}, we introduce a group $G$ that acts on every $\ph$-twisted $V_Q$-module. Imposing the $\ph$-equivariance condition for $\ph$-twisted modules, we define the group $G_\ph$ as a quotient of $G$ by a certain central subgroup.
We prove that, conversely, every module over the $\ph$-twisted Heisenberg algebra equipped with a compatible action of $G_\ph$ can be extended to a $\ph$-twisted $V_Q$-module 
 %This is the main result of the paper 
(see \thref{main}). 

Finally, in \seref{septvm}, we construct explicit examples from $3$ and $4$-dimensional lattices.

In a sequel we plan to classify the irreducible modules of the group $G_\ph$, thus completing the classification and explicit construction of all irreducible $\ph$-twisted $V_Q$-modules.

%%%%%%%%%%%%%%%%%%%%%%%%%%%%%%%%%%%%%%%%%%%%%%%%%
%%%%%%%%%%%%%%%%%%%%%%%%%%%%%%%%%%%%%%%%%%%%%%%%%
\section{Preliminaries}\label{sprelim}
%%%%%%%%%%%%%%%%%%%%%%%%%%%%%%%%%%%%%%%%%%%%%%%%%
%%%%%%%%%%%%%%%%%%%%%%%%%%%%%%%%%%%%%%%%%%%%%%%%%

In this section, we review the definitions and some properties of vertex algebras, lattice vertex algebras, and twisted logarithmic modules.

\subsection{Vertex Algebras}

Let us quickly recall the definition of a vertex (super)algebra following \cite{K2} (see also \cite{FLM2, FB, LL, KRR}).
A \emph{vertex algebra}  is a vector superspace $V=V_{\bar0} \oplus V_{\bar1}$ with a distinguished even vector $\vac\in V_{\bar0}$ 
(vacuum vector), together with a parity-preserving linear map (state-field correspondence)
\begin{equation*}%\label{vert2}
Y(\cdot,z)\cdot \colon V \otimes V \to V(\!(z)\!) = V[[z]][z^{-1}].
\end{equation*}
Thus, for every $a\in V$, we have the \emph{field} $Y(a,z) \colon V \to V(\!(z)\!)$. This field can be viewed as
a formal power series from $(\End V)[[z,z^{-1}]]$, which involves only finitely many negative powers of $z$ when applied to any vector.
The coefficients in front of powers of $z$ in this expansion are known as the
\emph{modes} of $a$:
\begin{equation*}%\label{vert4}
Y(a,z) = \sum_{n\in\ZZ} a_{(n)} \, z^{-n-1}, \qquad
a_{(n)} \in \End V.
\end{equation*}
Then
\begin{equation*}%\label{fres}
a_{(n)} = \res_z z^n Y(a,z),
\end{equation*}
where the formal residue $\res_z$ is defined as the coefficient of $z^{-1}$.

The following axioms must hold in a vertex algebra $V$:
\begin{equation*}%\label{vac}
\text{(vacuum axioms)} \quad
Y(\vac,z)=\Id, \quad Y(a,z)\vac-a \in zV[[z]],
\end{equation*}
where $\Id$ denotes the identity operator;
\begin{equation*}%\label{trasl}
\text{(translation covariance)} \quad
\bigl[T,Y(a,z)\bigr]=\partial_z Y(a,z),
\end{equation*}
where %$Ta = \partial_z Y(a,z)\vac |_{z=0}$;
$Ta=a_{(-2)}\vac$;
\begin{equation*}%\label{loc}
\text{(locality)} \quad
(z-w)^N Y(a,z) Y(b,w) = (-1)^{p(a)p(b)} (z-w)^N Y(b,w) Y(a,z)
\end{equation*}
for all $a,b\in V$ of parities $p(a),p(b)$, respectively, where $N\ge0$ is an integer depending on $a,b$.

%%%%%%%%%%%%%%%%%%%%%%%%%%%%%%%%%%%%%%%%%%%%%%%%%
%%%%%%%%%%%%%%%%%%%%%%%%%%%%%%%%%%%%%%%%%%%%%%%%%
\subsection{Lattice Vertex Algebras}\label{sslattice}
%%%%%%%%%%%%%%%%%%%%%%%%%%%%%%%%%%%%%%%%%%%%%%%%%
%%%%%%%%%%%%%%%%%%%%%%%%%%%%%%%%%%%%%%%%%%%%%%%%%

Now we recall the definition and some properties of lattice vertex algebras; see \cite{FLM2, K2, FB, LL}.
Let $Q$ be an \emph{integral lattice} of rank $d$, i.e., a free abelian group on $d$ generators together with 
a nondegenerate symmetric bilinear form $(\cdot | \cdot) \colon Q \times Q \to \ZZ$.
Using bilinearity, we extend $(\cdot | \cdot)$ to the vector space $\lie{h} = \CC \otimes_\ZZ Q$, which we view as an abelian Lie algebra.

The \emph{Heisenberg Lie algebra} is defined as the affinization $\hh = \lieh[t,t^{-1}]\oplus \CC K$ with brackets
\begin{equation}\label{sbbrackets}
[at^m,bt^n] = m \delta_{m,-n}(a|b)K, \qquad [\hh,K] = 0 \qquad (m,n \in \ZZ).
\end{equation}
We will use the notation $a_{(m)}=at^m$. Then the \emph{free bosons}
\begin{equation*}
a(z) = \sum_{m\in \ZZ}a_{(m)}z^{-m-1} \qquad (a \in \lieh)
\end{equation*}
satisfy
\begin{equation*}%\label{sbOPE}
[a(z),b(w)] = (a|b)K \, \partial_w \de(z,w),
\end{equation*}
where
\begin{equation}\label{delta}
\de(z,w) = \sum_{m \in\ZZ} z^{-m-1} w^{m}
\end{equation}
is the \emph{formal delta function}.

The generalized Verma module, also known as the \emph{Fock space},
\begin{equation*}%\label{bverma1}
B^1(\h) = \ind_{\h[t]\oplus \CC K}^{\hat{\h}} \CC
\end{equation*}
is constructed by letting $\h[t]$ act trivially on $\CC$ and $K$ act as $\Id$. Then $B^1(\h)$ has the structure of a vertex algebra called the \textit{free boson algebra} or the \emph{Heisenberg vertex algebra}.
The commutator \eqref{sbbrackets} is equivalent to the following $n$-th products:
\begin{align}\label{nthprodb}
a_{(0)}b &= 0, & a_{(1)}b &= (a|b)\vac, & a_{(j)}b &= 0 &&  (j \geq 2)
\end{align} 
for $a,b \in \h$, where $\vac=1$ is the vacuum vector in $B^1(\h)$.
Here and further, we identify $a\in\lieh$ with $a_{(-1)}\vac \in B^1(\h)$.

Now let $\ep\colon Q \times Q \to \{\pm 1\}$ be a bimultiplicative function such that
\begin{equation}\label{eplala}
\ep(\la,\la) = (-1)^{|\la|^2(|\la|^2+1)/2}, \qquad \la \in Q.
\end{equation}
A unique such function exists up to equivalence and satisfies
\begin{equation}\label{epprod}
\ep(\la,\mu)\ep(\mu,\la) = (-1)^{(\la |\mu)+|\la |^2|\mu |^2}, \qquad \la, \mu \in Q.
\end{equation}
Let $\CC_\ep[Q]$ be the twisted group algebra, which has a basis $\{e^\la\}_{\la \in Q}$ with multiplication given by 
\begin{equation}\label{expmult}
e^\la e^\mu = \ep(\la,\mu) e^{\la + \mu}, \qquad \la,\mu \in Q.
\end{equation}

The representation of $\hat\lieh$ can be extended to the space $V_Q =B^1(\h) \otimes \CC_\ep[Q]$ via
\begin{equation*}
(at^m)(s \otimes e^\la) = (at^m + \de_{m,0}(a |\la))s \otimes e^\la. %\qquad.
\end{equation*}
In particular, we have
\begin{equation}\label{azerola}
a_{(m)} e^\la = \de_{m,0} (a|\la) e^{\la}, \qquad a \in \lieh, \;  m\ge0, \; \la \in Q.
\end{equation}
The twisted group algebra $\CC_\ep[Q]$ can also be represented on $V_Q$ by the following action:
\begin{equation*}%\label{alaction}
e^\la(s \otimes e^\mu) = \ep(\la,\mu)s \otimes e^{\la+\mu}.
\end{equation*}

From now on we will write $e^\la$ (respectively, $a$) for $1 \otimes e^\la\in V_Q$ (respectively, $a \otimes 1)$. 
The fields
\begin{equation}\label{currents}
Y(a,z) = \sum_{m \in \ZZ}a_{(m)}z^{-m-1}, \qquad a \in \h,
\end{equation}
on $V_Q$ are called \emph{currents}. 
%The lattice vectors $\la \in Q$ act semisimply with integral eigenvalues on $V_Q$ via the action of the zero modes given by \eqref{azerola}.
%This allows us to define operators $z^\la$ by 
The fields 
\begin{equation}\label{voperators}
\begin{split}
Y(e^\la,z) &= e^\la \, \nop{\exp \int Y(\la,z)}\\
&=e^\la z^\la \exp \Bigl{(} \sum_{n=1}^\infty \la_{(-n)}\frac{z^{n}}{n}\Bigr{)}  
\exp \Bigl{(}\sum_{n=1}^\infty \la_{(n)}\frac{z^{-n}}{-n}\Bigr{)}
\end{split}
\end{equation}
are called \emph{vertex operators}, where $z^\la$ acts by
\begin{equation*}%\label{zexpla}
z^\la (s \otimes e^\mu )= z^{(\la|\mu)} (s \otimes e^\mu).
\end{equation*}

The currents \eqref{currents} and the vertex operators \eqref{voperators} generate a vertex algebra structure on $V_Q$, called a \emph{lattice vertex algebra}, where the vacuum vector is given by $1\otimes e^0$, the infinitesimal translation operator acts by
\begin{equation}\label{texp}
Te^\la = \la_{(-1)}e^\la, \qquad [T,a_{(m)}]=-m a_{(m-1)},
\end{equation}
and the parity in $V_Q$ is given by $p(a \otimes e^\la) = |\la|^2 \;\mathrm{mod}\; 2\ZZ$.
Reminiscent of \eqref{expmult}, the following formula holds in $V_Q$:
\begin{equation}\label{VAexpmult}
e^\la_{(-1-(\la|\mu))}e^\mu = \ep(\la,\mu)e^{\la+\mu}.
\end{equation}
The locality of the vertex operators $Y(e^\la,z)$ and $Y(e^\mu,z)$ is established by showing that
\begin{equation}\label{voploc}
\begin{split}
(z_1-z_2&)^N Y(e^\la,z_1)Y(e^\mu,z_2) \\
&= (-1)^{|\la|^2|\mu|^2}(z_1-z_2)^N Y(e^\mu,z_2)Y(e^\la,z_1),
\end{split}
\end{equation}
where $N=\max(0,-(\la|\mu))$.

The currents \eqref{currents} generate the Heisenberg vertex algebra $B^1(\lieh)$ as a subalgebra of $V_Q$.
Let $\{v_1,\ldots,v_d\}$ be a basis of $\lieh$, and $\{v^1,\ldots,v^d\}$ be its dual basis with respect to $(\cdot|\cdot)$.  Then
\begin{equation}\label{fbomega}
\om = \frac{1}{2}\sum_{i=1}^dv^i_{(-1)}v_i \in B^1(\lieh) \subset V_Q
\end{equation}
is a \emph{conformal vector} (see e.g.\ \cite{K2}). The corresponding Virasoro field
\begin{equation*}%\label{latticevir}
L(z) = Y(\om,z) = \sum_{n \in \ZZ}L_n z^{-n-2}
\end{equation*}
has central charge $d=\dim\lieh$. 
%For future use, we also recall that
%\begin{equation}\label{Lela}
%L_0 e^\la = \frac{|\la|^2}{2} e^\la, \qquad \la\in Q.
%\end{equation}

\subsection{Twisted Logarithmic Modules}

Now we review the definition and properties of twisted logarithmic modules of vertex algebras, following \cite{Bak}. The main distinguishing feature of such modules is the presence in the fields of a second formal variable, $\ze$, which plays the role of $\log z$. 

For a vector space $M$ over $\CC$, a \emph{logarithmic field} on $M$ is a formal series 
\begin{equation}\label{azze}
a(z) = a(z,\ze) = \sum_{\al \in \A} \sum_{m \in \al} a_m(\ze) z^{-m-1}, 
\end{equation}
where $\A$ is a finite subset of $\CC/\ZZ$, $a_m(\ze) \colon M \to M[[\ze]]$ is a power series in $\ze$ with coefficients in $\End M$, and for any $v \in M$ we have $a_m(\ze)v = 0$ for $\re m \gg 0$. Although $a(z)$ depends on both $z$ and $\ze$, for brevity we will write explicitly only the dependence on $z$.

While in \cite{Bak} the variables $z$ and $\ze$ were considered independent, here we will impose the relation that $z=e^\ze$. More precisely, we will identify a logarithmic field $a(z,\ze)$ with $z^c e^{-c\ze} a(z,\ze)$ for any $c\in\CC$, and extend this by linearity. In other words, if we denote the space of formal series \eqref{azze} by $\BLF(M)$, then the space of all logarithmic fields on $M$ is the quotient space 
\begin{equation*}%\label{LFie}
\LF(M) = \BLF(M) \Big/ \sum_{c\in\CC} (z^c-e^{c\ze}) \BLF(M).
\end{equation*}
Note that the operators of differentiation with respect to $z$ and $\ze$,
\begin{equation*}%\label{dzdze}
D_z = \partial_z + z^{-1}\partial_\ze, \qquad D_\zeta = z \partial_z + \partial_\ze,
\end{equation*}
commute with $z^c e^{-c\ze}$ and
are well defined on the quotient space $\LF(M)$.

Let $a(z), b(z) \in \LF(M)$ be logarithmic fields with parities $p(a)$ and $p(b)$, respectively. 
They are called \emph{local} if there exists an integer $N \geq 0$ such that
\begin{equation*}%\label{logloc}
(z_1-z_2)^Na(z_1)b(z_2)=(-1)^{p(a)p(b)}(z_1-z_2)^Nb(z_2)a(z_1).
\end{equation*}
Then their \emph{$n$-th product} is defined by \cite{Bak}:
\begin{equation}\label{TLMnthprod}
(a(z)_{(n)}b(z))v = \frac{D_{z_1}^{N-1-n}}{(N-1-n)!} \Bigl{(}(z_1-z_2)^N a(z_1)b(z_2)v\Bigr{)}\bigg{|}_{\substack{z_1=z_2=z \\\ze_1=\ze_2=\ze}},
\end{equation}
 for $v \in M$ and an integer $n\leq N-1$. For $n\geq N$, we set the $n$-th product equal to 0.
 Note that the right-hand side of \eqref{TLMnthprod} remains the same if we replace $N$ by $N+1$.

\begin{definition}[\hspace{1sp}\cite{Bak}]\label{defphtwisted}
Given a vertex algebra $V$ and an automorphism $\ph$ of $V$, a \emph{$\ph$-twisted $V$-module} is a vector superspace $M$ equipped with an even linear map $Y\colon V \to \LF(M)$
such that $Y(\vac,z) = \Id$ is the identity on $M$, and $Y(V)$ is a local collection of logarithmic fields satisfying:
\begin{equation}\label{phequiv}
\text{($\ph$-equivariance)} \quad
Y(\ph a,z) = e^{2\pi \ii D_\ze}Y(a,z),
\end{equation}
and
\begin{equation}\label{lognth}
\text{($n$-th product identity)} \quad
Y(a_{(n)}b,z) = Y(a,z)_{(n)}Y(b,z)
\end{equation}
for all $a,b \in V$ and $n \in \ZZ$.
\end{definition}

We note that $Y(a,z)$ depends on $\zeta$ as well as $z$, though we do not indicate this in the notation. 
As a consequence of the definition, the logarithmic fields satisfy (cf. \cite{H}):
\begin{equation}\label{infDz}
Y(Ta,z) = D_zY(a,z), \qquad a \in V.
\end{equation}

Let $\bar{V} \subseteq V$ be the subalgebra of $V$ on which $\ph$ is \emph{locally-finite}, i.e., the set of vectors $a \in V$ such that $\Span\{\ph^i a \,|\, i \geq 0\}$
is finite-dimensional. Then we can write
\begin{equation}\label{sigN}
\ph |_{\bar{V}} = \si e^{-2\pi \ii \N}, \qquad \sigma \in \aut(\bar{V}), \;\; \N \in  \der(\bar{V}),
\end{equation}
where %$\sigma \in \aut(\bar{V})$, $\N \in  \der(\bar{V})$, 
$\sigma$ and $\N$ commute, $\sigma$ is semisimple, and $\N$ is locally nilpotent on $\bar{V}$.
The last condition means that for every $a\in\bar V$ we have $\N^i a=0$ for $i\gg0$.
As in \cite{Bak}, we will assume that for $a\in\bar V$ the logarithmic field $Y(a,z)\in\LF(M)$ has a (unique) representative
$\bar Y(a,z)\in\BLF(M)$ that is a polynomial in $\ze$.
Then $\ph$-equivariance implies that 
\begin{equation}\label{tfieldx}
\bar Y(\si a,z) = e^{2\pi \ii z\partial_z} \bar Y(a,z), \qquad 
\bar Y(\N a,z) = -\partial_{\ze} \bar Y(a,z) %\qquad a\in\bar V.
\end{equation}
for $a\in\bar V$.
%We note that $\bar Y(a,z)$ is a polynomial of $\ze$ for $a \in \bar{V}$, since $\N$ is nilpotent on $a$.
%Moreover, $\bar Y(a,z)$ is the unique such representative of $Y(a,z)$.

\begin{remark}
Our current definition of $\LF(M)$ is more general than the definition in \cite{Bak} in order to %accommodate the need to 
allow logarithmic fields that are power series in $\ze$ when $\ph$ is not locally finite, as is the case for twisted logarithmic modules of lattice vertex algebras. In the case when $\ph$ is locally finite, we have $\bar{V}=V$. By choosing the representative $\bar{Y}(a,z)$ for each $Y(a,z) \in \LF(M)$, we may equate the more general definition we use in this paper to the original definition given in \cite{Bak}.
\end{remark}

\begin{remark}
Huang's definition of a twisted logarithmic module \cite{H} assumes, among other things, that the vertex algebra $V$ is graded by finite-dimensional subspaces preserved by the automorphism $\ph$. As a consequence, $\ph$ is locally finite in \cite{H}. It was shown in \cite{HY} that, under the additional assumptions of \cite{H}, the definitions of \cite{H} and \cite{Bak} are equivalent.
\end{remark}

We denote the eigenspaces of $\sigma$ in $\bar V$ by 
\begin{equation*}
V_\al = \{a \in \bar{V} \,|\, \sigma a = e^{-2\pi\ii \al} a\}, \qquad \al \in \CC/\ZZ.
\end{equation*}
It follows from \eqref{tfieldx} that 
\begin{equation*}%\label{tfieldx}
X(a,z) = \bar Y(e^{\zeta \N}a,z) = \bar Y(a,z)\big|_{\zeta = 0}
\end{equation*}
is independent of $\ze$, and the exponents of $z$ in $X(a,z)$ belong to $-\al$ for $a\in V_\al$.
%Whenever $a \in V_\al$ we will let $\al_0$ be the coset representative of $\al$ such that $-1<\re \al_0 \leq 0$.
For $m \in \al$, the \emph{$(m+\N)$-th mode} of $a\in V_\al$ is defined as
\begin{equation*}%\label{tmodes}
a_{(m+\N)} = \res_z z^m X(a,z).
\end{equation*}
%where, as usual, $\res_z$ denotes the coefficient of $z^{-1}$.
Then
\begin{equation}\label{tfield}
\bar Y(a,z) = X(e^{-\zeta \N}a,z) = \sum_{m \in \al} (z^{-m-\N-1}a)_{(m+\N)},
\end{equation}
where we use the notation $z^{-\N}=e^{-\ze \N}$.
%From now on we will just write $Y(a,z)$ instead of $\bar Y(a,z)$ for $a\in\bar V$. 

One of the main results of \cite{Bak} is the Borcherds identity for twisted logarithmic modules.
Here we will only need two of its consequences. One is the following \emph{commutator formula} for modes:
\begin{equation}\label{lcomm}
\bigl{[}a_{(m+\N)},b_{(n+\N)} \bigr{]}= \sum_{j=0}^\infty\Bigl{(} \bigl{(} \binom{m+\N}{j} a\bigr{)}_{(j)}b \Bigr{)}_{(m+n-j+\N)}
\end{equation}
for $a \in V_\al$, $b \in V_\be$, $m\in\al$, and $n\in\be$. We also have
\begin{equation}\label{lcommf}
\bigl{[}a_{(m+\N)},Y(b,z) \bigr{]}= \sum_{j=0}^\infty Y\Bigl{(} \bigl{(} \binom{m+\N}{j} z^{m-j+\N} a\bigr{)}_{(j)}b,z \Bigr{)}
\end{equation}
for $a \in V_\al$, $b \in V$, $m\in \al$.
Note that in \eqref{lcommf} the vector $b$ is not necessarily in $\bar V$.
This is important in the case of twisted logarithmic modules of lattice vertex algebras because for them $\bar V\neq V$ in general (see Remark \ref{notlocfin} below).

Another corollary of the Borcherds identity is a formula for the $(-1)$-st product.
For $a \in V_\al$ and $b \in V$, the \emph{normally ordered product} of $Y(a,z)$ and $Y(b,z)$ is defined by
\begin{align*}%\label{lnop}
\nop{Y(a,z) Y(b,z)} = &\sum_{\substack{m \in \al \\ \re m \leq -1}} (z^{-m-\N-1}a)_{(m+\N)} Y(b,z) \\
+ (-1)^{p(a)p(b)} &\sum_{\substack{m \in \al \\ \re m > -1}} Y(b,z) (z^{-m-\N-1}a)_{(m+\N)}.
\end{align*}
Then
\begin{equation}\label{twistednop}
\nop{Y(a,z)Y(b,z)} = \sum_{j=-1}^{N-1}z^{-j-1}Y\Bigl{(} \bigl{(} \binom{\al_0 + \N}{j + 1} a\bigr{)}_{(j)}b,z\Bigr{)}
\end{equation}
for $a \in V_\al$ and $b \in V$, where $\al_0\in\al$ is the unique element such that $-1 < \re \al_0 \leq 0$.

\subsection{Twisted Logarithmic Modules of Free Bosons}\label{ssbosons}

Consider a vector space $\lieh$ with a nondegenerate symmetric bilinear form $(\cdot | \cdot)$,
as in \seref{sslattice}.
Let $\ph$ be an automorphism of $\h$ such that $(\cdot|\cdot)$ is $\ph$-invariant, i.e.,
\begin{equation}\label{phinv}
(\ph a|\ph b) =(a|b), \qquad a,b \in \lieh.
\end{equation}
As in \eqref{sigN}, we write 
\begin{equation}\label{phsplit}
\ph=\si e^{-2 \pi \ii \N}, \qquad \si \N = \N \si,
\end{equation}
where $\si$ is semisimple and $\N$ is nilpotent on $\lieh$. 
Then \eqref{phinv} is equivalent to 
\begin{equation}\label{inv}
(\si a | \si b) = (a|b), \qquad (\N a|b) + (a | \N b) = 0
\end{equation}
for all $a,b \in \lieh$. We denote the eigenspaces of $\sigma$ by 
\begin{equation*}
\h_\al = \{a \in \h \,|\, \sigma a = e^{-2\pi\ii \al} a\}, \qquad \al \in \CC/\ZZ.
\end{equation*}

\begin{definition}[\hspace{1sp}\cite{Bak}]\label{def:twaff}
The \emph{$\ph$-twisted affinization} (or \emph{$\ph$-twisted Heisenberg algebra})
$\hhp$ is the Lie algebra spanned by a central element $K$ and elements $a_{(m+\N)}=at^m$ $(a \in \h_\al, \, m \in \al)$, with the Lie bracket
\begin{equation*}%\label{waffbrack}
[a_{(m+\N)},b_{(n+\N)}]=\delta_{m,-n}((m+\N)a|b)K
\end{equation*}
for $a \in \h_\al$, $b \in \h_\beta$, $m \in \al,$ $n\in \beta$.
\end{definition}

An $\hhp$-module $M$ is called \textit{restricted} if for every $a \in \lieh_\al$, $m \in \al$, $v \in M$, there is an integer $L$ such that $(at^{m+i})v =0$ for all $i \in \ZZ,$ $i \geq L$. We note that every highest weight $\hhp$-module is restricted (see \cite{K1}).
The automorphism $\ph$ naturally induces automorphisms of $\hat{\h}$ and $B^1(\h)$, which we will denote again by $\ph$.  Then every $\ph$-twisted $B^1(\lieh)$-module is a restricted $\hhp$-module and, conversely, every restricted $\hhp$-module uniquely extends to a $\ph$-twisted $B^1(\lieh)$-module \cite[Theorem 6.3]{Bak}. 

We split $\CC$ as a disjoint union of subsets $\CC^+$, $\CC^-=-\CC^+$ and $\{0\}$ where 
\begin{equation}\label{cplus}
\CC^+ = \{\gamma \in \CC \,|\, \re \gamma > 0 \} \cup \{ \gamma \in \CC \,|\, \re \gamma = 0, \, \im \gamma > 0\}. 
\end{equation}
Then the $\ph$-twisted Heisenberg algebra $\hat{\h}_\ph$ has a triangular decomposition 
\begin{equation}\label{wtriangle}
\hat{\h}_\ph = \hat{\h}_\ph^- \oplus \hat{\h}_\ph^0 \oplus \hat{\h}_\ph^+,
\end{equation}
 where
\begin{equation*}
\hat{\h}^\pm_\ph = \Span\{at^m \,|\, a \in \h_\al, \, \al \in \CC/\ZZ, \, m \in \al \cap \CC^\pm\}
\end{equation*}
and
\begin{equation*}
\hat{\h}_\ph^0 = \Span\{at^0 \,|\, a \in \h_0\}\oplus \CC K.
\end{equation*}
Starting from an $\hat{\h}^0_\ph$-module $R$ with $K=\Id$, the (generalized) \emph{Verma module} is defined by
\begin{equation*}%\label{bverma}
M_\ph(R) = \ind_{\hat{\h}_\ph^+\oplus \hat{\h}_\ph^0}^{\hat{\h}_\ph}R  ,
\end{equation*}
where $\hat{\h}_\ph^+$ acts trivially on $R$. These are $\ph$-twisted $B^1(\lieh)$-modules. 

In order to describe such modules explicitly, the following canonical forms for automorphisms $\ph$ of $\lieh$ preserving $(\cdot | \cdot )$ are used (see \cite{Bak}). 
In the next two examples, $\lieh$ is a vector space with a basis $\{v_1,\ldots,v_d\}$ such that $(v_i|v_j)=\de_{i+j,d+1}$ and 
$\la=e^{-2\pi\ii\al_0}$ for some $\al_0 \in \CC$ such that $-1 < \re \al_0 \leq 0$.

\begin{example}[$d = 2\ell$]\label{ex:symm1}
\begin{align*}
\sigma v_i & = \begin{cases}
\lambda v_i, & 1\leq i \leq \ell,\\
\lambda^{-1}v_i, & \ell+1 \leq i \leq 2\ell,
\end{cases} \quad
\N v_i = \begin{cases}
v_{i+1}, & 1 \leq i \leq \ell-1,\\
-v_{i+1}, & \ell + 1 \leq i \leq 2\ell -1, \\
0, & i = \ell,2\ell.
\end{cases}
\end{align*}
Let us write $\la^{-1}=e^{-2\pi\ii\be_0}$ where $-1 < \re \be_0 \leq 0$. The symmetry 
\begin{equation}\label{symmetry}
v_i \mapsto (-1)^iv_{\ell+i}, \qquad v_{\ell+i} \mapsto (-1)^{i+\ell+1}v_i \qquad (1\leq i \leq \ell)
\end{equation}
 allows us to switch $\la$ with $\la^{-1}$ and assume that $\al_0 \in \CC^- \cup \{0\}$. 
For example, if $\re \al_0 = 0$, then we may switch the roles of $\al_0$ and $\be_0$ if necessary to ensure that $\im \al_0 \leq 0$. 
If $-1 < \re \al_0 < 0$, then again employing the symmetry \eqref{symmetry}, we may assume $-1/2 \leq \re \al_0 < 0$. Finally, if $\re \al_0 = -1/2$, then \eqref{symmetry} allows us to assume $\im \al_0 \geq 0$.
\end{example}

\begin{example}[$d=2\ell-1$ and $\la=\pm 1$]\label{ex:symm2}
\begin{equation*}
\sigma v_i = \lambda v_i, \ \ 1 \leq i \leq 2\ell-1, \qquad \N v_i = \begin{cases}
(-1)^{i+1}v_{i+1}, &1 \leq i \leq 2\ell-2,\\
0, & i = 2\ell-1.
\end{cases}
\end{equation*}
Since $\la=\pm 1$, it follows that $\al_0 =0$ or $-1/2$. 
\end{example}

\begin{proposition}[\hspace{1sp}\cite{Bak}]\label{prop:symm}
Let\/ $\lieh$ be a finite-dimensional vector space endowed with a nondegenerate symmetric bilinear form\/ $(\cdot|\cdot)$ and with commuting linear operators\/ $\sigma$, $\N$ satisfying \eqref{inv}, such that\/ $\sigma$ is invertible and semisimple and\/ $\N$ is nilpotent. Then\/ $\lieh$ is an orthogonal direct sum of subspaces that are as in Examples \ref{ex:symm1} and \ref{ex:symm2}.
\end{proposition}

The $\ph$-twisted modules corresponding to Examples \ref{ex:symm1} and \ref{ex:symm2} are explicitly constructed in \cite{Bak}. In each case, the action of the Virasoro operators $L_k$ corresponding to the conformal vector \eqref{fbomega}
can be written explicitly using the following proposition. % (cf. \cite{BS}).
In it we will use the notation
\begin{equation}\label{nopm}
\nopm{(at^m)(bt^n)} = 
\begin{cases}
(at^m)(bt^n), & m \in \CC^-,\\
(bt^n)(at^m), & m \in \CC^+\cup\{0\}.
\end{cases}
\end{equation}

\begin{proposition}[\hspace{1sp}\cite{Bak,BS}]\label{prop:bvir}
If\/ $\ph$ is as in Example \ref{ex:symm1}, we have
\begin{equation*}%\label{bLk1}
L_k = \sum_{i=1}^\ell \sum_{m\in\al_0+\ZZ} \nopm{(v^it^{-m})(v_it^{k+m})}-\de_{k,0}\frac{\ell}{2}\al_0(\al_0+1) \,\Id
\end{equation*}
in any\/ $\ph$-twisted\/ $B^1(\lieh)$-module. In the case of Example \ref{ex:symm2}, we have
\begin{equation*}%\label{bLk2}
L_k = \frac12\sum_{i=1}^d \sum_{m\in\al_0+\ZZ} \nopm{(v^it^{-m})(v_it^{k+m})}-\de_{k,0}\frac{d}{4}\al_0(\al_0+1) \,\Id.
\end{equation*}
\end{proposition}

%%%%%%%%%%%%%%%%%%%%%%%%%%%%%%%%%%%%%%%%%%%%%%%%%
%%%%%%%%%%%%%%%%%%%%%%%%%%%%%%%%%%%%%%%%%%%%%%%%%
\section{The $\ph$-twisted Vertex Operators}\label{sphtwvo}
%%%%%%%%%%%%%%%%%%%%%%%%%%%%%%%%%%%%%%%%%%%%%%%%%
%%%%%%%%%%%%%%%%%%%%%%%%%%%%%%%%%%%%%%%%%%%%%%%%%

In this section, we start our investigation of the structure of $\ph$-twisted modules over a lattice vertex algebra $V_Q$, where the automorphism $\ph$ is lifted from an automorphism of the lattice $Q$. We find an expression for the twisted vertex operators $Y(e^\la,z)$ ($\la\in Q$) in terms of the twisted Heisenberg algebra.
We continue using the notation from \seref{sslattice}.
 
 %we let $Q$ be an integral lattice with a bilinear form $(\cdot|\cdot)$ and $\lieh = \CC \otimes_\ZZ Q$ be its complexification, and $S = \Vtilde^1$ be the Fock space representation of $\hh$. We let $\ep \colon Q\times Q \to \{\pm 1\}$ be a bimultiplicative function such that \eqref{eplala} holds and $V_Q = S \otimes \CC_\ep[Q]$ the corresponding lattice vertex algebra. 

\subsection{Action of the Heisenberg Subalgebra}
Let $\ph$ be an automorphism of the lattice $Q$, so that $\ph$ preserves the bilinear form (see \eqref{phinv}). We extend $\ph$ linearly to the vector space $\lieh = \CC \otimes_\ZZ Q$, denoting the extension again by $\ph$. 
Again we will write $\ph=\si e^{-2 \pi\ii\N}$ as in \eqref{phsplit}. 

The map $Q \times Q \to \{\pm 1\}$ given by $(\la,\mu) \mapsto \ep(\ph\la,\ph\mu)$ is a 2-cocycle satisfying \eqref{eplala}. Since $\ep$ is unique up to equivalence, there exists some $\eta\colon Q \to \{\pm 1\}$ such that 
\begin{equation}\label{etaep}
\eta(\la)\eta(\mu)\ep(\la,\mu) = \eta(\la+\mu)\ep(\ph\la,\ph\mu), \qquad \la,\mu \in Q.
\end{equation}
Then $\ph$ can be lifted to an automorphism of the lattice vertex algebra $V_Q$, denoted again by $\ph$, so that 
\begin{equation}\label{phext}
\ph(a_{(m)}) = (\ph a)_{(m)}, \qquad \ph(e^\la) = \eta(\la)^{-1}e^{\ph\la} %\qquad a \in \lieh, \ \la \in Q, \ m \in \ZZ
\end{equation}
for $a \in \lieh$, $\la \in Q$, $m \in \ZZ$
(see e.g.\ \cite{BK}). %Proposition 4.1).

\begin{remark}\label{notlocfin} %In general,
If $\ph$ has an infinite order when acting on $\la\in Q$, then the vectors $\ph^i(e^\la)$, $i\ge 0$, are linearly independent in $V_Q$. This means that the action of $\ph$ on $e^\la$ is not locally finite and the decomposition \eqref{phsplit} is not valid for such vectors. However, \eqref{phext} guarantees that $\ph$ is locally finite on the Heisenberg subalgebra $B^1(\lieh) \subset V_Q$.
\end{remark}

Assume that $M$ is a $\ph$-twisted $V_Q$-module. Then the logarithmic fields $Y(a,z)$, $a \in \lieh$, generate a $\ph$-twisted $B^1(\lieh)$-module structure on $M$.
Let $S_\si$ be the spectrum of $\si$, i.e., its set of eigenvalues, and let 
\begin{equation*}
\A = \{\al \in \CC / \ZZ \,|\, e^{-2\pi \ii \al} \in S_\si\}.
\end{equation*}
Given $\al \in \A$, we let $\pi_\al \colon \lieh \to \lieh_{\al}$ be the projection onto the corresponding eigenspace. 
%Then we can write any $\la \in Q$ uniquely as a sum of eigenvectors for $\si$:
%\begin{equation*}
%\la = \sum_{\al \in \A} \pi_\al \la.
%\end{equation*}
We adopt the convention that 
\begin{equation}\label{modeproj}
a_{(m+\N)} = (\pi_\al a)_{(m+\N)}, \qquad m \in \al,
\end{equation}
for any $a \in \lieh$ and $\al \in \A$.  
In particular, $a_{(m+\N)}=0$ if $m\in\al$ for $\al \not\in \A$. We will also use the notation $\pi_0= \pi_\ZZ$, $\lieh_0 = \lieh_\ZZ$, and $a_0 = \pi_0 a$.

Since $\lieh_\al \perp \lieh_\be$ for $\be \neq -\al$, we have
\begin{equation}\label{bilproj}
(\pi_\al a |b ) = (\pi_\al a | \pi_{-\al} b ), \qquad \al \in \A, \;\; a,b \in \lieh.
\end{equation}
Then from \eqref{nthprodb}, \eqref{lcomm}, \eqref{modeproj}, we obtain
\begin{equation}\label{hbrext}
\bigl{[}a_{(m+\N)},b_{(n+\N)}\bigr{]} = \bigl{(}(m+\N)\pi_{\al}a \big{|}b \bigr{)}\de_{m,-n}, \qquad m \in \al, \ n \in \CC.
\end{equation}
Let $\la \in Q$ and $\al \in \A$.  Then \eqref{azerola}  and \eqref{lcommf}
imply
\begin{align}\label{hvobrext}
\bigl{[}a_{(m+\N)},Y(e^\la,z)\bigr{]} &= \bigl{(}z^{m+\N} \pi_{\al}a \big{|}\la \bigr{)}Y(e^\la,z), \qquad a \in \lieh, \ m \in \al.
\end{align}

\subsection{The Exponentials $E_\la(z)$}

As usual, the solution of \eqref{hvobrext} is obtained in terms of the exponential of $\phi_{\la}(z)$, 
where $\phi_{\la}(z)$ is a logarithmic field satisfying
$D_z \phi_\la(z) = Y(\la,z)$. To find such a field, recall that
\begin{equation}\label{lafield}
Y(\la,z) =\sum_{\al \in \A}\sum_{m \in \al}\bigl{(}z^{-m-1-\N} \la\bigr{)}_{(m+\N)} =\sum_{m \in \CC}\bigl{(}z^{-m-1-\N} \la\bigr{)}_{(m+\N)}.
\end{equation}
For $m \in \CC \setminus \{0\}$, the operator $(m+\N) \colon \lieh \to \lie h$ is invertible on $\lieh$ with
\begin{equation}\label{mpni}
(m+\N)^{-1}a = \sum_{j=0}^\infty \frac{(-\N)^j}{m^{j+1}}a.
\end{equation}
Recall that the operator $z^{-\N}$ acts on $\lieh$ as
\begin{equation}\label{zN}
z^{-\N}a = e^{-\ze \N}a = \sum_{j=0}^\infty \frac1{j!} (-\ze)^j \N^ja.
\end{equation}
The sums in both \eqref{mpni} and \eqref{zN} are finite, since $\N$ is a nilpotent operator on $\lieh$.

We define linear operators $\P^\pm = \P_\ze^\pm \colon \lieh \to \lieh[\ze]$ by 
\begin{equation}\label{Ppm}
\P^\pm a 
%= \int_0^\ze e^{\pm\ze \N}a \, ds  
=  \pm\frac{1}{\ze\N}\bigl{(}z^{\pm\N}-1 \bigr{)} a
=  \pm\frac{1}{\ze\N}\bigl{(}e^{\pm\ze\N}-1 \bigr{)} a.
\end{equation}
These operators commute with $\si$ and hence preserve its eigenspaces $\lieh_\al$. Note that by \eqref{inv},
\begin{equation*}%\label{Ppmbil}
(\P^+ a| b) = (a | \P^- b), \qquad a,b \in \lieh.
\end{equation*}
We also define the operator $\P = \P_\ze=(\P^+ - \P^-)/2$, which satisfies
\begin{equation}\label{Pskew}
(\P a | b) = -(a | \P b).
\end{equation}

Then the logarithmic field
\begin{equation*}
\phi_{\la}(z) = -\sum_{m \in \CC \setminus\{0\}}\bigl{(}(m+\N)^{-1}z^{-m-\N}\la \bigr{)}_{(m+\N)}+(\ze \P^-\la)_{(0+\N)}
\end{equation*}
satisfies $D_z \phi_\la(z) = Y(\la,z)$.
Now we consider the exponentials
\begin{equation*}%\label{Elazpm}
E_\la(z)_{\pm} = \exp \Bigl{(} - \sum_{m \in \CC^\mp} \bigl{(} (m+\N)^{-1}z^{-m-\N}\la\bigr{)}_{(m+\N)}\Bigr{)},
\end{equation*}
and define
 \begin{equation*}%\label{Elaz}
E_\la(z) = E_\la(z)_+E_\la(z)_-.
\end{equation*}

\begin{proposition}\label{emodes}
Let\/ $a \in \lieh$ and\/ $m \in \al$. Then 
\begin{equation*}
\bigl{[} a_{(m+\N)},E_\la(z)\bigr{]} = 
\begin{cases}
\bigl{(}z^{m+\N} \pi_\al a |\la\bigr{)} E_\la(z), & m \neq 0,\\
0, & m = 0.
\end{cases}
\end{equation*}
\end{proposition}
\begin{proof}
The case when $m=0$ follows immediately from \eqref{hbrext}. If $m \neq 0$, we use \eqref{hbrext} and the fact that $\N$ satisfies \eqref{inv} to compute
\begin{align*}
\bigl{[} a_{(m+\N)},E_\la(z)\bigr{]} & = - \bigl{[} a_{(m+\N)},(-m+\N)^{-1}(z^{m-\N}\la)_{(-m+\N)}\bigr{]} E_\la(z)\\
& =  -\bigl{(}(m+\N)\pi_\al a |(-m+\N)^{-1}z^{m-\N}\la\bigr{)}E_\la(z) \\
& =  \bigl{(}z^{m+\N} \pi_\al a |\la\bigr{)}E_\la(z).
\end{align*}
\end{proof}

\subsection{The Operators $U_\la(z)$}

For $h\in\lieh$, we define the operator
\begin{equation}\label{thetah}
\begin{split}
\th_h &= \th_h(\ze) = e^{(\ze\P^- h)_{(0+\N)}} = \exp \Bigl(\frac{1-e^{-\ze\N}}\N h\Bigr)_{(0+\N)} \\
&= \exp \Bigl( \ze h - \frac{\ze^2}{2!} \N h + \frac{\ze^3}{3!} \N^2 h - \cdots \Bigr)_{(0+\N)}.
\end{split}
\end{equation}

\begin{lemma}\label{lthmult}
The operators\/ \eqref{thetah} satisfy
\begin{equation}\label{thmult}
\th_{h}\th_{h'} = e^{(\ze\P h_0 | h')}\th_{h+h'} = e^{2(\ze\P h_0|h')}\th_{h'}\th_{h}
\end{equation}
for\/ $h,h' \in \lieh$. In particular, $\th_h$ is invertible with\/ $\th_h^{-1} = \th_{-h}$.
\end{lemma}
\begin{proof}
Using \eqref{hbrext} and \eqref{inv}, we compute:
\begin{align*}
[(\ze\P^-h)_{(0+\N)} &,(\ze\P^-h')_{(0+\N)}] 
= \Bigl( (1-e^{-\ze\N}) h_0 \Big| \frac1\N (1-e^{-\ze\N}) h' \Bigr) \\
&= -\Bigl( \frac1\N(1-e^{-\ze\N}) h_0 \Big| h' \Bigr) + \Bigl( \frac1\N e^{\ze\N} (1-e^{-\ze\N}) h_0 \Big| h' \Bigr) \\
&= 2(\ze\P h_0| h').
\end{align*}
This proves \eqref{thmult}. 
The fact that  $\th_{h} \th_{-h} = \th_0$ follows from \eqref{Pskew}.
\end{proof}

Now we define the operators
\begin{equation}\label{ulambda}
U_\la(z)  = E_\la(z)_+^{-1} Y(e^\la,z) \th_\la^{-1} E_\la(z)_-^{-1}, \qquad \la \in Q.
\end{equation}

\begin{proposition}\label{umodes}
For\/ $a \in \lieh$ and\/ $\la \in Q$, we have 
\begin{equation*}
[a_{(m+\N)},U_\la(z)] = \de_{m,0} (a_0 | \la) U_\la(z).
\end{equation*}
\end{proposition}
\begin{proof}
First, as in the proof of \prref{emodes}, we have for $m \neq 0$
\begin{equation*}%\label{intbrack}
\Bigl{[} a_{(m+\N)}, \bigl{(} (n+\N)^{-1}z^{-n-\N}\la\bigr{)}_{(n+\N)}\Bigr{]}\\
%&= \de_{m,-n}\Bigl{(} (m+\N)\pi_\al a \Big{|}(-m+\N)^{-1}z^{m-\N}\la\Bigr{)}\\
= - \de_{m,-n}\bigl{(}z^{m+\N} \pi_\al a \big{|} \la \bigr{)}.
\end{equation*}
From here and \eqref{hvobrext} we deduce that $\bigl{[} a_{(m+\N)}, U_\la(z)\bigr{]} = 0$. Now assume that $m=0$. Then 
\begin{align*}
\bigl{[} a_{(0+\N)} &, U_\la(z) \bigr{]} \\
&=\bigl{(} z^\N a_0 \big{|} \la \bigr{)}U_\la(z)+ \biggl{[}a_{(0+\N)}, \Bigl{(}\frac{1}{\N}\bigl{(}z^{-\N}-1 \bigr{)}\la \Bigr{)}_{(0+\N)} \biggr{]}U_\la(z)\\
& = \bigl{(}z^\N a_0 \big{|} \la \bigr{)} U_\la(z)+\Bigl{(} \N a_0 \Big{|} \frac{1}{\N}\bigl{(} z^{-\N}-1\bigr{)} \la \Bigr{)}U_\la(z) \\
& =\bigl{(}z^\N a_0 \big{|}\la\bigr{)} U_\la(z) -\bigl{(}(z^\N-1) a_0 | \la \bigr{)} U_\la(z) \\
& = (a_0 | \la) U_\la(z),
\end{align*}
completing the proof.
\end{proof}

Using \eqref{ulambda} and \prref{umodes}, we obtain
\begin{equation}\label{expfieldla}
Y(e^\la,z) = E_\la(z)_+ U_\la(z)E_\la(z)_- \,\th_{\la} = U_\la(z)\th_{\la}E_\la(z).
\end{equation}
%using that $[E_\la(z)_+,U_\la(z)]=0$.
On the other hand, from \eqref{texp} and \eqref{infDz} we have 
\begin{equation}\label{Dm1}
D_zY(e^\la,z) = Y(\la_{(-1)}e^\la,z).
\end{equation}
We will use this equality to produce a partial differential equation in $z$ and $\zeta$ satisfied by $U_\la(z)$. In order to do so, we need two lemmas.
\begin{lemma}\label{derlemma}
Let\/ $A$ be an associative algebra and\/ $D \colon A \to A$ a derivation. Assume\/ $X \in A$ satisfies\/ $[X,D(X)] = C$, where\/ $C$ commutes with\/ $X$. Then for any\/ $n \in \NN$ we have
\begin{equation}
\begin{split}
\label{derpower}
D(X^n) &= nX^{n-1}D(X)-\frac{n(n-1)}{2}CX^{n-2}\\
&=nD(X)X^{n-1}+\frac{n(n-1)}{2}CX^{n-2}.
\end{split}
\end{equation}
Furthermore, in any representation of\/ $A$ on which\/ $X$ can be exponentiated, we have
\begin{equation}\label{derexp}
D(e^X) = e^X \Bigl{(}D(X)-\frac{C}{2}\Bigr{)}
=\Bigl{(}D(X)+\frac{C}{2}\Bigr{)} e^X .
\end{equation}
\end{lemma}
\begin{proof}
This lemma is a special case of the Baker--Campbell--Hausdorff formula. For completeness, we include a simple proof of \eqref{derpower} by induction on $n$. The claim is clearly true for $n=0$ and $n=1$. Now assume it is true for $n-1$. Then 
\begin{align*}
D&(X^n) = D(X^{n-1})X + X^{n-1}D(X) \\ 
& = \Bigl{(}(n-1)X^{n-2}D(X)- \frac{(n-1)(n-2)}{2}CX^{n-3}\Bigr{)}X + X^{n-1}D(X) \\
&=(n-1)X^{n-2}(XD(X)-C)-\frac{(n-1)(n-2)}{2}CX^{n-2}+X^{n-1}D(X)\\
&=nX^{n-1}D(X)-\frac{n(n-1)}{2}CX^{n-2}.
\end{align*}
Similarly by shifting $D(X)$ to the left instead of right we obtain the second line of \eqref{derpower}.
Then we prove \eqref{derexp} using \eqref{derpower}:
\begin{align*}
D(e^X) %& = \sum_{n=0}^\infty \frac{D(X^n)}{n!}\\
& = \sum_{n=0}^\infty\frac{1}{n!}\Bigl{(} nX^{n-1}D(X)-\frac{n(n-1)}{2}CX^{n-2}\Bigr{)}\\
&=\sum_{n=1}\frac{X^{n-1}}{(n-1)!}D(X)-\frac{C}{2}\sum_{n=2}^\infty\frac{X^{n-2}}{(n-2)!}\\
& = e^X\Bigl{(}D(X)-\frac{C}{2}\Bigr{)}.
\end{align*}
The second line of \eqref{derexp} is derived similarly using the second line of \eqref{derpower}.
\end{proof}

As in \cite{Bak}, introduce an operator $\mathcal{S}\colon \lieh \to \lieh$ by $\mathcal{S} a = \al_0a$ for $a \in \lieh_\al$, $\al \in\A$, and extend linearly to all of $\lieh$. Here, as before, $\al_0\in\al$ is the unique element such that $-1 < \re \al_0 \leq 0$. We also define
\begin{align*}
\A^+& = \{\al \in \A  \,|\, \al_0 \in \CC^+\}\\
&=\{\al \in \A \,|\, \re \al_0 =0 \,\text{ and }\, \im \al_0 > 0\},\\
\intertext{and}
\A^-& = \{\al \in \A  \,|\, \al_0 \in \CC^-\}\\
&=\{\al \in \A \,|\, \re \al_0 <0 \,\text{ or }\, \re \al_0 = 0 \,\text{ and }\, \im \al_0 < 0\}.
\end{align*}
Then $\A=\A^+ \cup \A^- \cup \{\ZZ\}$.

Let us also introduce the notation
\begin{equation}\label{bla}
b_\la = \frac{|\la_0|^2-|\la|^2}{2} \,, \qquad \la\in\lieh, \; \la_0=\pi_0\la,
\end{equation}
and recall that the normally ordered product $\nopm{\;}$ is given by \eqref{nopm}.

\begin{lemma}\label{Dzlem}
For any\/ $\la \in Q$ we have 
\begin{equation*}
D_z Y(e^\la,z) = \nopm{Y(\la,z)Y(e^\la,z)}+z^{-1}b_\la Y(e^\la,z).
\end{equation*}
\end{lemma}
% \begin{lemma}\label{m1lambda}
% With the above notation, for any\/ $\la \in Q$ we have
% \begin{equation*}
% Y(\la_{(-1)}e^\la,z) = \nop{Y(\la,z)Y(e^\la,z)} - z^{-1}\bigl{(}\mathcal{S}\la \big{|}\la\bigr{)}Y(e^\la,z)
% \end{equation*}
% and
% \begin{equation*}%\label{sla}
% (\mathcal{S}\la|\la)=-\frac{1}{2}|\la|^2 - \frac{1}{2}|\la_0|^2+\sum_{\mathclap{\al \in \A^+ \cup \, \ZZ}} \,(\pi_\al \la | \la).
% \end{equation*}
% \end{lemma}
\begin{proof}
From \eqref{azerola} and \eqref{twistednop}, we obtain
\begin{equation*}%\label{nopcalc}
\begin{split}
\nop{Y(\la &,z)Y(e^\la,z)} = \sum_{\al \in \A}\nop{Y(\pi_\al \la, z) Y(e^\la,z)} \\
& = \sum_{\al \in \mathcal{A}}Y\bigl{(}(\pi_\al\la)_{(-1)}e^\la,z\bigr{)} 
+ \sum_{\al \in \mathcal{A}}z^{-1}Y\bigl{(}((\al_0 + \N)\pi_\al\la)_{(0)}e^\la,z\bigr{)}\\
& = Y(\la_{(-1)}e^\la,z) + z^{-1}\bigl{(}(\mathcal{S}+\N)\la\big{|}\la\bigr{)}Y(e^\la,z)\\
&=Y(\la_{(-1)}e^\la,z) + z^{-1} (\mathcal{S}\la | \la) Y(e^\la,z).
\end{split}
\end{equation*}
In the last line we used that $(\N \la | \la) = 0$ by \eqref{inv}. 

From \eqref{bilproj} and the symmetry of $(\cdot | \cdot)$ we have
\begin{equation}\label{projflip}
(\pi_\al \la | \la) = (\pi_{-\al}\la | \la).
\end{equation}
We also observe that 
\begin{equation}\label{al0cases}
(-\al)_0 = \begin{cases}
-\al_0, & \re \al_0 = 0, \\
-\al_0-1, & \re \al_0 < 0.
\end{cases}
\end{equation}
Hence for $\al\neq \ZZ, -\frac{1}{2}+\ZZ$, we have
\begin{equation*}%\label{projop}
(\al_0\pi_\al\la | \la) +((-\al)_0 \pi_{-\al}\la | \la) =
\begin{cases}
0, & \re \al_0 = 0, \\
-(\pi_\al\la | \la), & \re \al_0 < 0.
\end{cases} 
\end{equation*}

In the case of $\al_0 = -1/2$, we have $(-\al)_0 = \al_0$. Thus %from \eqref{projflip} and \eqref{projop} we obtain
\begin{equation*}%\label{Slambda}
\begin{split}
(\mathcal{S} \la | \la) & = \sum_{\al \in \A}(\al_0\pi_\al \la | \la) \\
& =-\frac{1}{2}(\pi_{1/2+\ZZ}\la|\la)+ \sum_{\al \in \A \setminus\{1/2+\ZZ\}}\,\frac{1}{2}\bigl{(}(\al_0\pi_\al \la | \la)+((-\al)_0\pi_{-\al}\la|\la)\bigr{)}\\
&= -\frac{1}{2}\sum_{\substack{\al \in \A \\ \re \al_0 < 0}} \, (\pi_\al\la | \la).
%& = -\frac{1}{2}\sum_{\al \in \A}(\pi_\al \la | \la) -\frac{1}{2}|\la_0|^2+\sum_{\mathclap{\al\in \A^+ \cup \, \{\ZZ\}}} \,(\pi_\al \la|\la) \\
%&= - \frac{1}{2}|\la|^2-\frac{1}{2}|\la_0|^2 +\sum_{\mathclap{\al \in \A^+ \cup \, \{\ZZ\}}} \,(\pi_\al \la|\la).
\end{split}
\end{equation*}
The final summation can alternatively be indexed by the complement of $\{\al \in \A \,| \re \al_0 = 0\}$ in $\A$, and from \eqref{projflip} we have 
\begin{equation*}
\begin{split}
\frac{1}{2}\sum_{\substack{\al \in \A \\ \re \al_0 = 0}} \, (\pi_\al\la | \la) &= \frac{1}{2}|\la_0|^2 + \sum_{\al \in \A^+}  (\pi_\al \la|\la) \\
&= - \frac{1}{2}|\la_0|^2+\sum_{\mathclap{\al \in \A^+ \cup \, \{\ZZ\}}} \, (\pi_\al \la|\la).
\end{split}
\end{equation*}
Thus 
\begin{equation}\label{Dzeq1}
\begin{split}
(\mathcal{S}\la|\la) &= -\frac{1}{2}\sum_{\al \in \A}(\pi_\al \la|\la)+\frac{1}{2}\sum_{\substack{\al \in \A \\ \re \al_0 = 0}} \, (\pi_\al\la | \la) \\
&= - \frac{1}{2}|\la|^2-\frac{1}{2}|\la_0|^2 +\sum_{\mathclap{\al \in \A^+ \cup \, \{\ZZ\}}} \,(\pi_\al \la|\la).
\end{split}
\end{equation}
%\end{proof}
% \begin{lemma}\label{nopdiff}
% For any\/ $\la \in Q$ we have
% \begin{equation*}
% \begin{split}
% \nop{Y(\la,z)Y(e^\la,z)} &= \nopm{Y(\la,z)Y(e^\la,z)} 
% - z^{-1}|\la|^2 \, Y(e^\la,z)\\
% &+z^{-1}\sum_{\mathclap{\al \in \mathcal{A}^+ \cup \, \{\ZZ\}}}\, (\pi_\al \la | \la) \, Y(e^\la,z).
% \end{split}
% \end{equation*}
% \end{lemma}
%\begin{proof} 

Noting that 
\begin{align*}
\nop{Y&(\la,z)Y(e^\la,z)} \\
& = \sum_{\mathclap{\substack{m\in \CC \\ \re m \leq -1}}}\bigl{(}z^{-m-1-\N}\la\bigr{)}_{(m+\N)}Y(e^\la,z)
 + Y(e^\la,z) \sum_{\mathclap{\substack{m\in \CC\\ \re m > -1}}}\bigl{(}z^{-m-1-\N}\la\bigr{)}_{(m+\N)}\\
&= \nopm{Y(\la,z)Y(e^\la,z)}
+ \sum_{\al \in \A^-}\Bigl{[}Y(e^\la,z),\bigl{(}z^{-\al_0-1-\N}\la\bigr{)}_{(\al_0+\N)} \Bigr{]},
\end{align*}
and that by \eqref{hvobrext} 
\begin{equation*}
\begin{split}
\sum_{\al \in \A^-}\Bigl{[}(z^{-\al_0-1-\N}\la&)_{(\al_0+\N)} ,Y(e^\la,z) \Bigr{]} = z^{-1} \sum_{\al \in \A^-}(\pi_\al \la | \la)Y(e^\la,z)\\
&= z^{-1}|\la|^2 \, Y(e^\la,z) -z^{-1}\sum_{\mathclap{\al \in \mathcal{A}^+ \cup \, \{\ZZ\}}}\, (\pi_\al \la | \la) \, Y(e^\la,z),
\end{split}
\end{equation*}
we obtain
\begin{equation}\label{Dzeq2}
\begin{split}
\nop{Y(\la,z)Y(e^\la,z)} &= \nopm{Y(\la,z)Y(e^\la,z)} 
- z^{-1}|\la|^2 \, Y(e^\la,z)\\
&+z^{-1}\sum_{\mathclap{\al \in \mathcal{A}^+ \cup \, \{\ZZ\}}}\, (\pi_\al \la | \la) \, Y(e^\la,z).
\end{split}
\end{equation}
The desired result follows immediately from \eqref{Dm1}, \eqref{Dzeq1}, and \eqref{Dzeq2}.
\end{proof}

\subsection{Twisted Vertex Operators}\label{subtvo}

In order to state the main result of this section, we need to introduce some additional notation. For any $\la\in\lieh$, we let
\begin{equation*}
\begin{split}
a_\la &= a_\la(\ze)= \frac{(\P^- \la_0 | \la) - |\la_0|^2}{2},\\
% b_\la &=\frac{|\la_0|^2-|\la|^2}{2}, \\
 c_\la &= 2\pi \ii \, a_\la(2\pi\ii) + b_\la = \frac{(\P^-_{2\pi\ii} \la_0 | \la) - |\la|^2}{2}, \\
 \tau_{\la} &= \th_\la(2\pi\ii),
\end{split}
 \end{equation*}
where $\P^-$ is given by \eqref{Ppm}, $\th_\la$ by \eqref{thetah}, $b_\la$ by \eqref{bla}, and as before $\la_0 = \pi_0\la$.

\begin{remark}\label{rphss1}
When the automorphism $\ph$ is semisimple, we have $\ph=\si$ and $\N=0$. Then $\P^\pm=\Id$, $\th_\la=e^{\ze\la_0}=z^{\la_0}$,
$\tau_\la=e^{2\pi\ii\la_0}$, $a_\la=0$, and $c_\la=b_\la$. In this case, the result of the next theorem reduces to \cite[Lemma 4.1]{BK}.
\end{remark}

\begin{theorem}
Assume that\/ $M$ is a\/ $\ph$-twisted\/ $V_Q$-module.
Then there exist operators\/ $U_\la$ $(\la \in Q)$ on\/ $M$, independent of\/ $z$ and\/ $\ze$, such that
\begin{equation}\label{twistedvop}
Y(e^\la,z) = U_\la \th_{\la} e^{\ze a_\la}z^{b_\la}E_\la(z),
\end{equation}
and
\begin{align}
[a_{(m+\N)},U_\la] &= \de_{m,0}(\pi_0 a | \la) U_\la, \label{aUcomm}\\
U_{\ph\la}& =\eta(\la) e^{2\pi\ii c_\la} U_\la \tau_\la,\label{phieq}
\end{align}
for\/ $a \in \lieh,$ $m \in \CC$.
\end{theorem}

\begin{proof}
Recalling \eqref{ulambda} and \eqref{expfieldla}, in order to prove \eqref{twistedvop}, we will first compute $D_z U_\la(z)$.
Using \leref{derlemma} with $D=D_z$, $X = (\ze\P^- \la)_{(0+\N)}$, 
\begin{equation*}
D(X) = \bigl{(}D_z(\ze \P^-(\la))\bigr{)}_{(0+\N)} = (z^{-\N-1}\la)_{(0+\N)},
\end{equation*}
and
\begin{equation*}
C = [(\ze\P^-\la)_{(0+\N)},(z^{-\N-1}\la)_{(0+\N)}]=z^{-1}\bigl{(}(z^{-\N}-1)\la_0 | \la \bigr{)},
\end{equation*}
we obtain 
\begin{equation*}%\label{dzph}
D_z\th_{\la}=\th_{\la}\Bigl{(}(z^{-\N-1}\la)_{(0+\N)}-\frac{1}{2}z^{-1}\bigl{(} (z^{-\N}-1)\la_0\big{|}\la\bigr{)} \Bigr{)}.
\end{equation*}
Next we compute
\begin{equation*}%\label{dze}
\begin{split}
D_zE_\la(z)&= \Bigl{(}\sum_{m \in \CC^-}(z^{-m-\N-1}\la)_{(m+\N)}\Bigr{)}E_\la(z)\\
&\qquad+ E_\la(z)\Bigl{(}\sum_{m \in \CC^+}(z^{-m-\N-1}\la)_{(m+\N)}\Bigr{)}\\
&=\nopm{\bigl{(}Y(\la,z)-(z^{-1-\N}\la)_{(0+\N)}\bigr{)}E_\la(z)},
\end{split}
\end{equation*}
%Thus from \eqref{dzph} and \eqref{dze} we have 
which gives
\begin{equation}\label{dzyel1}
\begin{split}
D_zY(e^\la,z) = \nopm{Y(\la,z)&Y(e^\la,z)} + (D_zU_\la(z))\th_{\la}E_\la(z)\\
&-\frac{1}{2}z^{-1}\bigl{(}(z^{-\N}-1)\la_0 | \la\bigr{)}Y(e^\la,z).
\end{split}
\end{equation}
%On the other hand, from \eqref{Dm1} and Lemmas \ref{m1lambda}, \ref{nopdiff} we have
% \begin{equation}\label{dzyel2}
% \begin{split}
% D_zY(e^\la,z) = \nopm{Y(\la,z)Y(e^\la,z)}+z^{-1}b_\la Y(e^\la,z).
% \end{split}
% \end{equation}
By comparing \eqref{dzyel1} with \leref{Dzlem}, we find that $U_\la(z)$ satisfies the following partial differential equation in $z$ and $\ze$:
\begin{equation*}
D_zU_\la(z) = \frac{1}{2}z^{-1}\bigl{(}(z^{-\N}-1)\la_0 \big{|} \la\bigr{)}U_\la(z) +z^{-1} b_\la U_\la(z).
\end{equation*}
This implies that the operator
\begin{equation}\label{Uops}
U_\la = z^{-b_\la}e^{-\ze a_\la}U_\la(z)
\end{equation}
satisfies $D_zU_\la = 0$, which means that $U_\la$ is a function of $z e^{-\ze}$. 
Since we quotient by the relations $z^c = e^{c\ze}$ $(c\in\CC)$, we can assume that $U_\la$
is independent of $z$ and $\ze$.

Now \eqref{aUcomm} follows immediately from \prref{umodes}.
To finish the proof of the theorem, it remains to prove \eqref{phieq}.
Applying \eqref{phequiv} and \eqref{phext}, we get
\begin{equation}\label{compare}
e^{2\pi\ii D_\ze}Y(e^\la,z) = \eta(\la)^{-1}Y(e^{\ph \la},z).
\end{equation}
It is easy to check that 
\begin{equation*}
\begin{split}
e^{2\pi \ii D_\ze}\th_{\la} &= e^{\frac{1}{2}(\ze \P^-(1-\ph)\la_0|\la)}\tau_\la\th_{\ph \la},\\
e^{2\pi \ii D_\ze}e^{\ze a_\la}&=e^{\frac{1}{2}(\ze \P^-(\ph-1)\la_0|\la)+2\pi \ii a_\la(2\pi \ii)}e^{\ze a_\la}, \\
e^{2\pi \ii D_\ze}E_\la(z) &= E_{\ph \la}(z).
\end{split}
\end{equation*}
Hence
\begin{equation*}%\label{comp1}
e^{2\pi \ii D_\ze}Y(e^\la,z) = e^{2\pi\ii c_\la}U_\la\tau_\la \th_{\ph \la }e^{\ze\al_\la}z^{b_\la}E_{\ph \la}(z).
\end{equation*}
On the other hand,
\begin{equation*}%\label{comp2}
Y(e^{\ph \la},z) = U_{\ph \la}\th_{\ph \la}e^{\ze a_\la}z^{b_\la}E_{\ph \la}(z).
\end{equation*}
Substituting the last two expressions into \eqref{compare},
we obtain \eqref{phieq}.
\end{proof}

%%%%%%%%%%%%%%%%%%%%%%%%%%%%%%%%%%%%%%%%%%%%%%%%%
%%%%%%%%%%%%%%%%%%%%%%%%%%%%%%%%%%%%%%%%%%%%%%%%%
\section{Products of Twisted Vertex Operators}\label{sptvo}
%%%%%%%%%%%%%%%%%%%%%%%%%%%%%%%%%%%%%%%%%%%%%%%%%
%%%%%%%%%%%%%%%%%%%%%%%%%%%%%%%%%%%%%%%%%%%%%%%%%

Our next goal is to find the relations satisfied by the operators $U_\la$. To this end, we will express the product
\begin{equation*}
(z_1-z_2)^{-(\la |\mu)}E_\la(z_1)E_\mu(z_2)\big{|}_{z_1=z_2=z}
\end{equation*} 
in terms of the exponential $E_{\la + \mu}(z)$.  
Here and further, we will expand $(z_1-z_2)^n$ as a formal power series in the domain $|z_1|>|z_2|$, namely,
\begin{equation*}
(z_1-z_2)^n = \sum_{j=0}^\infty \binom{n}{j} z_1^{n-j} (-z_2)^j.
\end{equation*}

\subsection{Special Functions}
We will make use of several well-known special functions, which we briefly review following \cite{BE1,BE2}.
The first of these is the \emph{Lerch transcendent}, a complex-valued function of three complex variables defined by the convergent series
\begin{equation}\label{lerch}
\Phi(z,s,a) = \sum_{m=0}^\infty \frac{z^m}{(m+a)^s}, \qquad |z|<1, \;\; a\neq 0,-1,-2,\ldots.
\end{equation}
The integral formula
\begin{equation}\label{lerchint}
\Phi(z,s,a) = \frac{1}{\Gamma(s)} \int_{0}^\infty\frac{t^{s-1}e^{-at}}{1-ze^{-t}} \, dt
\end{equation}
analytically extends $\Phi$ to include the cases:
\begin{enumerate}
\item[1.] $\re a > 0$, $\re s >0$, $z \in \CC \setminus[1,\infty)$, 
\item[2.] $\re a > 0$, $\re s >1$, and $z=1$.
\end{enumerate}
Throughout the rest of this section, when dealing with $\ln(z)$ we will assume $-\pi < \arg(z) < \pi$. For $|z|>1$, $z \ \notin (-\infty,1)\cup(1,\infty)$, and $a \notin \ZZ$, the Lerch transcendent $\Phi(z,1,a)$ admits the following convergent expansion (see \cite{FKS}):
\begin{equation}\label{lerchref}
\begin{split}
\Phi(& z,1,a) = \pi \ii z^{-a}\sgn\omega(z)+z^{-a}\pi\cot(\pi a)+z^{-1}\sum_{m=0}^\infty\frac{z^{-m}}{m+1-a} \\
&= \pi \ii z^{-a}\sgn\omega(z)+z^{-a}\pi\cot(\pi a)+z^{-1} \Phi(z^{-1},1,1-a),
\end{split}
\end{equation}
where $\omega(z) = \arg(\ln z)$.
We will also use the following obvious identities:
\begin{equation}\label{lerchshift}
\Phi(z,s,a) = z^n\Phi(z,s,a+n)+\sum_{k=0}^{n-1}\frac{z^k}{(k+a)^s} \,,
\end{equation}
and
\begin{equation}\label{lerchder}
\partial_a \Phi(z,s,a) = -s\Phi(z,s+1,a).
\end{equation}

Other functions that can be derived from $\Phi$ are relevant to this text, including the \emph{polylogarithm}
\begin{equation*}
\Li_s(z) = z \Phi(z,s,1) = \sum_{n=1}^\infty \frac{z^n}{n^s}
\end{equation*}
and the \emph{Riemann zeta-function}
\begin{equation*}
\ze(s) = \Phi(1,s,1) = \sum_{n=1}^\infty \frac{1}{n^s} \,.
\end{equation*}
We also need the closely related \emph{digamma function}:
\begin{equation*}%\label{digamma}
\Psi(a) = -\gamma + \sum_{m=0}^\infty\Bigl{(}\frac{1}{m+1}-\frac{1}{m+a}\Bigr{)},\qquad a\neq 0,-1,-2,\ldots,
\end{equation*}
where $\gamma \approx 0.577216$ is the Euler--Mascheroni constant. The digamma function satisfies the following \emph{reflection formula}:
\begin{equation}\label{digref}
\Psi(-a)-\Psi(a+1) = \pi \cot \pi a = 2 \pi \ii \Li_0(e^{-2\pi \ii a}) +\pi \ii, %\qquad a \in \CC \setminus\ZZ.
\end{equation}
for $a \in \CC \setminus\ZZ$.
The derivatives of the digamma function are the \emph{polygamma functions}:
\begin{equation}\label{polygamma}
\begin{split}
\Psi^{(j)}(a) &= \partial^j_a \Psi(a) = (-1)^{j+1}j!\sum_{m=0}^\infty\frac{1}{(m+a)^{j+1}} \\
&= (-1)^{j+1}j! \Phi(1,j+1,a),
%\qquad j\in \NN, \quad -a\notin\NN.
\end{split}
\end{equation}
defined for $j\in \NN$ and $a\notin -\NN$.
The polygamma functions also satisfy a reflection formula obtained by taking the $j$th derivative of \eqref{digref} with respect to $a$:
\begin{equation}\label{polyref}
(-1)^j\Psi^{(j)}(-a)-\Psi^{(j)}(a+1)=-(-2\pi \ii)^{j+1}\Li_{-j}(e^{-2\pi\ii a}), %\qquad a \in \CC \setminus\ZZ.
\end{equation}
for $a \in \CC \setminus\ZZ$.

The polygamma functions are related to the Riemann zeta-function as follows:
\begin{equation}\label{polyzeta}
\ze(j+1) = \frac{(-1)^{j+1}}{j!}\Psi^{(j)}(1), \qquad j=1,2,3,\ldots.
\end{equation}
The values of the Reimann zeta-function for nonnegative even integers can be obtained using the generating function
\begin{equation}\label{zegen}
\sum_{j=0}^{\infty}\ze(2 j)x^{2j} = -\frac{1}{2}\pi x\cot(\pi x).
\end{equation}

\subsection{Products of Exponentials}
Recall that for $\al\in\CC/\ZZ$, we define $\al_0\in\al$ to be the unique element such that $-1 < \re \al_0 \leq 0$. 
Let 
\begin{equation}\label{alprime}
\al_0'=\begin{cases}
\alpha_0+1, & \al_0 \in \CC^- \cup \{0\},\\
\alpha_0, & \al_0 \in \CC^+,
\end{cases}
\end{equation}
and define the operator $\mathcal{S}'\colon \lieh \to \lieh$ by letting $\mathcal{S}'a = \al_0'a$ for $a \in \al$, $\al \in \A$ and extending linearly.  
We also introduce the operator $\Psi(\mathcal{S'} + \N) \colon \lieh\to \lieh$ via the following Taylor series expansion:
\begin{equation}\label{seriesexp}
\begin{split}
\Psi(\mathcal{S'}+ \N)\la & =e^{\N \partial_z}\Psi(z)\big{|}_{z=\mathcal{S}'}\la
%&=\sum_{j=0}^\infty\frac{1}{j!}\partial_z^j\Psi(z)\big{|}_{z=\mathcal{S}'}\N^j\la\\
=\sum_{j=0}^\infty\frac{1}{j!}\Psi^{(j)}(\mathcal{S}')\N^j\la\\
&=\sum_{\al \in \A}\sum_{j=0}^\infty\frac{1}{j!}\Psi^{(j)}(\al_0')\N^j\pi_\al\la.
\end{split}
\end{equation}

\begin{proposition}\label{ecomb} 
For\/ $\la,\mu \in Q$, we have
\begin{equation*}%\label{twoEs}
\begin{split}
(z_1-z_2)^{-(\la|\mu)}E_\la(z_1)&E_\mu(z_2)\big{|}_{z_1=z_2=z} = z^{-(\la|\mu)} B_{\la,\mu}\,E_{\la+\mu}(z),
\end{split}
\end{equation*}
where 
\begin{equation}\label{Blambda}
B_{\la,\mu} = \exp \bigl{(} (\Psi(\mathcal{S}'+\N)+\ga)\la\big{|} \mu \bigr{)}
\end{equation}
and\/ $\gamma$ is the Euler--Mascheroni constant. 
\end{proposition}

\begin{proof}
First note that using %\eqref{polygamma} and 
\eqref{seriesexp} we can rewrite \eqref{Blambda} as 
\begin{equation*}%\label{Blamu}
B_{\la,\mu} = \exp\sum_{\al \in \A}\sum_{j=0}^\infty c_{\al,j}\bigl{(}\N^j \pi_\al \la \big{|} \mu\bigr{)},
\end{equation*}
where the $c_{\al,j}$ are constants given by
\begin{equation*}%\label{Eprodc}
c_{\al,j}=
\frac{1}{j!}\Psi^{(j)}(\al_0')+\de_{j,0}\ga.
\end{equation*}

We have:
\begin{equation}\label{elzemz}
\begin{split}
E_\la(z_1)E_\mu(z_2) &= E_\la(z_1)_+E_\la(z_1)_-E_\mu(z_2)_+E_\mu(z_2)_-\\
& = \exp(C)E_\la(z_1)_+E_\mu(z_2)_+E_\la(z_1)_-E_\mu(z_2)_- ,
\end{split}
\end{equation}
where $C$ is the commutator
\begin{equation*}
\Bigl{[}\sum_{m \in \CC^+} \bigl{(} (m+\N)^{-1}z_1^{-m-\N}\la\bigr{)}_{(m+\N)},\sum_{n \in \CC^-} \bigl{(} (n+\N)^{-1}z_2^{-n-\N}\mu\bigr{)}_{(n+\N)} \Bigr{]}.
\end{equation*}
Then using \eqref{hbrext} and \eqref{inv}, we write $C$ as
\begin{equation}\label{commsum}
C=\sum_{\al \in \A}\sum_{j=0}^\infty(-1)^{j+1}\Bigl{(}\Bigl{(}\frac{z_2}{z_1}\Bigr{)}^\N\N^j\pi_\al \la \Big{|} \mu \Bigr{)} \sum_{m \in \al^+}\frac{1}{m^{j+1}}\Bigl{(}\frac{z_2}{z_1}\Bigr{)}^m,
\end{equation}
where $\al^+ = \al \cap \CC^+$ and the sum over $j$ is finite due to the nilpotency of $\N$. 

For each $\al \in \A$ and $j \geq 1$, the sum
\begin{equation*}
\begin{split}
\sum_{m \in \al^+}\frac{1}{m^{j+1}}\Bigl{(}\frac{z_2}{z_1}\Bigr{)}^m 
&=\sum_{m=0}^\infty \frac{1}{(m+\al_0')^{j+1}}\Bigl{(}\frac{z_2}{z_1}\Bigr{)}^{m+\al_0'} \\
&=\Bigl{(}\frac{z_2}{z_1}\Bigr{)}^{\al_0'}\Phi\Bigl{(}\frac{z_2}{z_1},j+1,\al_0' \Bigr{)}
\end{split}
\end{equation*}
converges for $|z_1| \geq |z_2|$, and from \eqref{polygamma} we obtain
\begin{equation*}%\label{phitoze}
\Bigl{(}\frac{z_2}{z_1}\Bigr{)}^{\al_0'}\Phi\Bigl{(}\frac{z_2}{z_1},j+1,\al_0' \Bigr{)}\Big{|}_{z_1=z_2=z} =
\frac{(-1)^{j+1}}{j!}\Psi^{(j)}(\al_0')
= (-1)^{j+1} c_{\al,j}.
\end{equation*}
Now we consider the case when $j=0$. Then the sum is
\begin{equation*}%\label{jzero}
\sum_{m \in \al^+}\frac{1}{m}\Bigl{(}\frac{z_2}{z_1}\Bigr{)}^m = \Bigl{(}\frac{z_2}{z_1}\Bigr{)}^{\al'_0}\Phi\Bigl{(}\frac{z_2}{z_1},1,\al_0' \Bigr{)}, \qquad |z_1|\geq |z_2|, \ z_1 \neq z_2.
\end{equation*}
We cannot set $z_1=z_2$ in this expression, because both the sum \eqref{lerch} and the integral formula \eqref{lerchint} for $\Phi$ diverge.

Let $t=(\pi_\al \la | \mu)$. For $|z_1| >  |z_2|$, we have
\begin{equation*}%\label{explog}
\begin{split}
(z_1-z_2)^{-t}
%=z_1^{-t}\Bigl{(}1-\frac{z_2}{z_1}\Bigr{)}^{-t}
=z_1^{-t}\exp\Bigl{(}-t\ln\Bigl{(}1-\frac{z_2}{z_1}\Bigr{)}\Bigr{)}
=z_1^{-t}\exp \sum_{m=1}^\infty \frac{t}{m}\Bigl{(}\frac{z_2}{z_1}\Bigr{)}^m,
\end{split}
\end{equation*}
and
\begin{equation*}%\label{expsum}
\begin{split}
(&z_1-z_2)^{-t} \exp\Bigl{(} -\Bigl{(}\Bigl{(}\frac{z_2}{z_1}\Bigr{)}^\N\pi_\al \la \Big{|} \mu\Bigr{)} \Bigl{(}\frac{z_2}{z_1}\Bigr{)}^{\al_0'} \Phi\Big{(}\frac{z_2}{z_1},1,\al_0'\Big{)}\Bigr{)} \\
&= z_1^{-t} \exp\sum_{m=0}^\infty \Bigl{(}\frac{t}{m+1}\Bigl{(}\frac{z_2}{z_1}\Bigr{)}^{m+1} 
- \Bigl{(}\Bigl{(}\frac{z_2}{z_1}\Bigr{)}^\N\pi_\al \la \Big{|} \mu\Bigr{)}  \frac{1}{m+\al_0'}\Bigl{(}\frac{z_2}{z_1}\Bigr{)}^{m+\al_0'} \Bigr{)}.
\end{split}
\end{equation*}
The sum in the right-hand side converges for $z_1=z_2$ to
\begin{equation*}%\label{phitopsi}
\begin{split}
%\sum_{m=0}^\infty \Bigl{(}t\frac{1}{m+1}&\Bigl{(}\frac{z_2}{z_1}\Bigr{)}^{m+1} 
%- \bigl{(}\Bigl{(}\frac{z_2}{z_1}\Bigr{)}^\N\pi_\al \la \big{|} \mu\bigr{)}  \frac{1}{m+\al_0'}\Bigl{(}\frac{z_2}{z_1}\Bigr{)}^{m+\al_0'} \Bigr{)}\Big{|}_{z_1=z_2=z} \\
\sum_{m=0}^\infty \Bigl{(}\frac{t}{m+1} 
- \frac{t}{m+\al_0'} \Bigr{)}
=t(\Psi(\al_0')+\ga) = c_{\al,0} (\pi_\al \la | \mu).
\end{split}
\end{equation*}
Plugging these into \eqref{commsum}, we get
\begin{equation*}
\begin{split}
(z_1&-z_2)^{-(\la|\mu)}\exp(C)\big{|}_{z_1=z_2=z} \\
 %&=\prod_{\al \in \A}\biggl{[}(z_1-z_2)^{-(\pi_\al \la|\mu)}
%\exp\Bigl{(}-\Bigl{(}\Bigl{(}\frac{z_2}{z_1}\Bigr{)}^\N\pi_\al \la \Big{|} \mu \Bigr{)} \sum_{m \in \al^+}\frac{1}{m}\Bigl{(}\frac{z_2}{z_1}\Bigr{)}^m\Bigr{)}\\
%&\times \exp\Bigl{(}\sum_{j=1}^\infty(-1)^{j-1}\Bigl{(}\Bigl{(}\frac{z_2}{z_1}\Bigr{)}^\N\N^j\pi_\al \la \Big{|} \mu \Bigr{)} \sum_{m \in \al^+}\frac{1}{m^{j+1}}\Bigl{(}\frac{z_2}{z_1}\Bigr{)}^m\Bigr{)}\biggr{]}\bigg{|}_{z_1=z_2=z}\\
&= \prod_{\al \in \A}\biggl{(}z^{-(\pi_\al \la|\mu)}
\exp \bigl( c_{\al,0}(\pi_\al \la | \mu) \bigr)
\exp\sum_{j=1}^\infty c_{\al,j}\bigl{(}\N^j\pi_\al \la \big{|} \mu \bigr{)}\biggr{)}\\
&= \ z^{-(\la|\mu)}\exp \sum_{\al \in \A}\sum_{j=0}^\infty c_{\al,j}\bigl{(}\N^j \pi_\al \la \big{|} \mu\bigr{)},
\end{split}
\end{equation*}
completing the proof.
\end{proof}

\begin{corollary}
In every\/ $\ph$-twisted\/ $V_Q$-module, we have  
\begin{equation}\label{fieldprod}
\begin{split}
(z_1-z_2)^{-(\la|\mu)}&Y(e^\la,z_1)Y(e^\mu,z_2)\big{|}_{z_1=z_2=z} \\
&= B_{\la,\mu}U_\la U_\mu \theta_{\la+\mu}e^{\ze a_{\la+\mu}}z^{b_{\la+\mu}}E_{\la+\mu}(z)
\end{split}
\end{equation}
for\/ $\la,\mu \in Q$.
\end{corollary}
\begin{proof}
First, we compute 
\begin{equation*}%\label{alamu}
\begin{split}
a_{\la + \mu} &= a_\la+a_\mu + \Bigl{(}\frac{\P^++\P^-}{2}\la_0 \Big{|} \mu\Bigr{)}-(\la_0|\mu), \\
%\label{blamu}
b_{\la+\mu} &= b_\la + b_\mu + (\la_0|\mu) - (\la|\mu).
\end{split}
\end{equation*}
Using \eqref{aUcomm} we obtain
\begin{equation}\label{thU}
\th_\la U_\mu = e^{(\ze\P^-\la_0|\mu)}U_\mu\th_\la.
\end{equation}
Combining the above equations, \eqref{thmult}, and \prref{ecomb}, we obtain the desired result.
\end{proof}

\subsection{Products of Operators $U_\la$}

 We recall that $\lieh_0$ is the subset of elements of $\lieh$ that are fixed by $\si$. The operator $(1-\si)$ is invertible on $(\lieh_0)^\perp$, 
and every $\la\in\lieh$ can be written uniquely in the form
\begin{equation}\label{lasum}
\la = \la_0 + (1-\si)\la_*, \qquad \la_0=\pi_0\la\in\lieh_0, \;\; \la_* \in (\lieh_0)^\perp.
\end{equation}

\begin{theorem}\label{commutant}
For every\/ $\ph$-twisted\/ $V_Q$-module, we have  
\begin{equation}\label{Uprod}
U_\la U_\mu = \ep(\la,\mu)B^{-1}_{\la,\mu}U_{\la+\mu} = C_{\la,\mu}U_\mu U_\la,
\end{equation}
where\/ $B_{\la,\mu}$ is given by \eqref{Blambda} and
\begin{equation}\label{cla}
\begin{split}
C_{\la,\mu} = (-1)^{|\la|^2|\mu|^2}e^{\pi \ii (\la_0|\mu)}e^{2 \pi \ii (\frac{1-\si}{1-\ph}\la_* | \mu)}e^{(\frac{1}{\N}(\pi \ii \N \frac{1+\ph}{1-\ph}-1)\la_0|\mu)}.
\end{split}
\end{equation}
\end{theorem}

\begin{proof}
Let $N\ge\max(0,-(\la|\mu))$ be such that the locality \eqref{voploc} holds for the logarithmic vertex operators $Y(e^\la,z)$ and $Y(e^\mu,z)$. 
We apply \eqref{TLMnthprod} for $Y(e^\la,z)$ and $Y(e^\mu,z)$ with $n=-1-(\la|\mu)$. Then from \eqref{fieldprod} we obtain
\begin{equation}\label{VOPprod1}
\begin{split}
Y(e^\la_{(-1-(\la|\mu))}e^\mu,z) &= (z_1-z_2)^{-(\la|\mu)}Y(e^\la,z_1)Y(e^\mu,z_2) \big{|}_{z_1=z_2=z} \\
&=B_{\la,\mu}U_\la U_\mu \th_{\la+\mu}e^{\ze a_{\la+\mu}}z^{b_{\la+\mu}}E_{\la+\mu}(z).
\end{split}
\end{equation}
On the other hand, from \eqref{VAexpmult} we have 
\begin{equation}\label{VOPprod2}
\begin{split}
Y(e^\la_{(-1-(\la|\mu))}e^\mu,z)  &= \ep(\la,\mu)Y(e^{\la+\mu},z) \\
&=\ep(\la,\mu)U_{\la+\mu}\th_{\la+\mu}e^{\ze a_{\la+\mu}}z^{b_{\la+\mu}}E_{\la+\mu}(z).
\end{split}
\end{equation}
Comparing \eqref{VOPprod1} and \eqref{VOPprod2}, we get the first equality of \eqref{Uprod}.

To prove \eqref{cla}, first observe that by \eqref{al0cases} and \eqref{alprime},
\begin{equation*}
(-\al)_0' = -\al_0'+1, \qquad \al \in \A \setminus \{\ZZ\}.
\end{equation*}
Then from \eqref{inv}, \eqref{bilproj}, we get
\begin{equation*}
((1-\si)\la_*|\mu) = ((1-\si)\la_*|(1-\si)\mu_*) = (\la|(1-\si)\mu_*)
\end{equation*}
and
\begin{equation}\label{al0pr}
\bigl( (1-\mathcal{S}'-\N)(1-\si)\la_* \big| \mu \bigr) = \bigl( \la \big| (\mathcal{S}'+\N)(1-\si)\mu_* \bigr).
\end{equation}
This implies
\begin{equation*}%\label{Ceq1}
\bigl{(}\Psi(\mathcal{S}'+\N)\mu \big{|}(1-\si)\la_* \bigr{)} = \bigl{(}\Psi(-\mathcal{S}'-\N+1)(1-\si)\la_* \big{|} \mu  \bigr{)}.
\end{equation*}
From the reflection formulas \eqref{polyref}, we obtain
\begin{equation*}
\begin{split}
(\Psi(-\mathcal{S}'&-\N+1)-\Psi(\mathcal{S}'+\N))(1-\si)\la_*\\
&=e^{\N \partial_z}(\Psi(-z)-\Psi(z+1))\big{|}_{z=\mathcal{S}'-1}(1-\si)\la_*\\
&=e^{\N \partial_z}(\pi \cot(\pi z))\big{|}_{z=\mathcal{S}'-1}(1-\si)\la_*\\
&=\pi \cot(\pi (\mathcal{S}'+\N))(1-\si)\la_*.
\end{split}
\end{equation*}
Using that 
\begin{equation*}%\label{cotan}
\cot(z) = \ii \frac{e^{\ii z}+e^{-\ii z}}{e^{\ii z}-e^{-\ii z}}=\ii\frac{1+e^{-2\ii z}}{1-e^{-2\ii z}} 
=\frac{2\ii}{1-e^{-2\ii z}} - \ii
\end{equation*}
and
\begin{equation*}
\ph = e^{-2\pi\ii (\mathcal{S}+\N)} = e^{-2\pi\ii (\mathcal{S}'+\N)},
\end{equation*}
we find 
\begin{equation*}
\begin{split}
\bigl{(}\pi \cot(\pi (\mathcal{S}'+\N)) (1-\si)\la_* \big{|} \mu \bigr{)}
%&=\pi \ii \Bigl{(}\frac{1+e^{-2\pi\ii (\mathcal{S}'+\N)}}{1-e^{-2\pi\ii (\mathcal{S}'+\N)}} (1-\si)\la_*\Big{|}\mu\Bigr{)}\\
%&=\pi\ii  \Bigl{(}\frac{1+\ph}{1-\ph} (1-\si)\la_* \Big{|} \mu \Bigr{)}\\
&= 2\pi \ii \Bigl{(}\frac{1-\si}{1-\ph}\la_*\Big{|}\mu \Bigr{)} -\pi\ii ((1-\si)\la_*|\mu)\\
&= 2\pi \ii \Bigl{(}\frac{1-\si}{1-\ph}\la_*\Big{|}\mu \Bigr{)}-\pi\ii (\la-\la_0|\mu).
\end{split}
\end{equation*}

Next we note that $\mathcal{S}'$ acts as the identity on $\lieh_0$. Thus by \eqref{polyzeta}, \eqref{zegen}, \eqref{seriesexp}, we have
\begin{equation*}
\begin{split}
\bigl((\Psi(\mathcal{S}'+\N)&+\ga)\mu\big|\la_0\bigr)-\bigl((\Psi(\mathcal{S}'+\N)+\ga)\la_0\big|\mu\bigr) \\
&=\bigl((\Psi(1-\N)-\Psi(1+\N))\la_0\big|\mu\bigr)\\
%&= \bigl{(}\bigl{(}e^{-\N\partial_z}\Psi(z)-e^{\N\partial_z}\Psi(z) \bigr{)}\big{|}_{z=1}\la_0 \big{|} \mu\bigr{)}\\
&= \sum_{j=0}^\infty  \Bigl{(}(-1)^j\frac{\Psi^{(j)}(1)}{j!}- \frac{\Psi^{(j)}(1)}{j!}\Bigr{)} \bigl{(}\N^j\la_0 \big{|} \mu\bigr{)}\\
&= -2 \sum_{j=0}^\infty \ze(2 j+2)\bigl{(} \N^{2j+1}\la_0 \big{|} \mu \bigr{)}\\
&= \Bigl{(} \frac{1}{\N} (\pi \N \cot(\pi \N)-1) \la_0 \Big{|}\mu\Bigr{)}\\
&= \Bigl{(}\frac{1}{\N} \Bigl{(}\pi \ii \N \frac{1+\ph}{1-\ph}-1\Bigr{)} \la_0 \Big{|}\mu\Bigr{)}.
\end{split}
\end{equation*}
Therefore
\begin{equation*}
\begin{split}
(-1&)^{|\la|^2|\mu|^2} C_{\la,\mu} = (-1)^{(\la|\mu)}\frac{B_{\mu,\la}}{B_{\la,\mu}}\\
&=
 (-1)^{(\la|\mu)} \exp\Bigl{(}\bigl{(} (\Psi(\mathcal{S}'+\N)+\ga)\la\big{|} \mu \bigr{)} -\bigl{(} (\Psi(\mathcal{S}'+\N)+\ga)\mu\big{|} \la \bigr{)}  \Bigr{)}\\
&=(-1)^{(\la|\mu)}e^{2\pi \ii (\frac{1-\si}{1-\ph}\la_* |\mu )} e^{-\pi\ii (\la-\la_0|\mu)}e^{(\frac{1}{\N}(\pi\ii\N\frac{1+\ph}{1-\ph}-1)\la_0|\mu)}\\
&=e^{\pi \ii (\la_0|\mu)}e^{2\pi \ii (\frac{1-\si}{1-\ph}\la_* |\mu )}e^{(\frac{1}{\N}(\pi\ii\N\frac{1+\ph}{1-\ph}-1)\la_0|\mu)}.
\end{split}
\end{equation*}
This completes the proof of the theorem.
\end{proof}

\begin{remark}\label{rphss2}
When the automorphism $\ph$ is semisimple, we have $\ph=\si$ and $\N=0$. Then 
\begin{equation*}
\frac{1}{\N} \Bigl{(}\pi \ii \N \frac{1+\ph}{1-\ph}-1\Bigr{)} \la_0 
= -2 \sum_{j=0}^\infty \ze(2 j+2) \N^{2j+1}\la_0 = 0.
\end{equation*}
In this case, \eqref{cla} reduces to \cite[Eq.\ (4.44)]{BK}. 
When $\ph=\si$ has finite order, our $B_{\la,\mu}$ coincide with those given by \cite[Eq.\ (4.37)]{BK}. 
\end{remark}

%%%%%%%%%%%%%%%%%%%%%%%%%%%%%%%%%%%%%%%%%%%%%%%%%
%%%%%%%%%%%%%%%%%%%%%%%%%%%%%%%%%%%%%%%%%%%%%%%%%
\section{Reduction to Group Theory}\label{srtgt}
%%%%%%%%%%%%%%%%%%%%%%%%%%%%%%%%%%%%%%%%%%%%%%%%%
%%%%%%%%%%%%%%%%%%%%%%%%%%%%%%%%%%%%%%%%%%%%%%%%%

As before, let $Q$ be an integral lattice with an automorphism $\ph$.
Building on the results of the previous section, here we will introduce a group $G$ that acts on
every $\ph$-twisted $V_Q$-module.
Conversely, every module over the $\ph$-twisted Heisenberg algebra equipped with an action of $G$ satisfying suitable $\ph$-equivariance and compatibility conditions
can be extended to a $\ph$-twisted $V_Q$-module.

\subsection{The Group $\hat G$}\label{subtgg}

Recall that $\lieh=\CC\otimes_\ZZ Q$ and we write $\ph|_\lieh=\si e^{-2 \pi\ii\N}$ (see \eqref{phsplit}).
We let $\lieh_0$ be the subspace of vectors in $\lieh$ fixed by $\si$, and define $B_{\la,\mu}$ by \eqref{Blambda}.

\begin{definition}\label{defG}
Let $\hat G = \CC^\times \times Q \times \exp(\lieh_0)$
 be the set of elements of the form $c \,U_\la e^{h}$ ($c \in \CC^\times$, $\la \in Q$, $h \in \lieh_0$),
%\begin{equation}\label{groupel}
%c \,U_\la e^{h} \qquad (c \in \CC^\times, \ \la \in Q, \ h \in \lieh_0),
%\end{equation}
with the multiplication determined by
\begin{equation}\label{multinG}
\begin{split}
e^{h} e^{h'} &= e^{\frac{1}{2}(\N h|h')}e^{h+ h'},\\
e^{h} U_\la e^{-h} &= e^{(h | \la)}U_\la,\\
U_\la U_\mu &=\ep(\la,\mu)B_{\la,\mu}^{-1}U_{\la+\mu}.
\end{split}
\end{equation}
\end{definition}

More explicitly, using \eqref{inv}, we can write 
\begin{equation}\label{Gmultip}
(c U_\la e^{h})(c' U_\mu e^{h'}) = c c' e^{(h|\mu-\frac{1}{2}\N h')}\ep(\la,\mu)B_{\la,\mu}^{-1}U_{\la+\mu}e^{h+h'}.
\end{equation}

\begin{proposition}
The set\/ $\hat G$ with the multiplication defined above is a group with identity element\/
$\mathbbm{1}= 1  U_0 e^0.$
\end{proposition}
\begin{proof}
It is easy to check from \eqref{Gmultip} that $c U_\la e^h$ is invertible with
\begin{equation*}%\label{Ginvert}
(c U_\la e^{h})^{-1} =c^{-1} \ep(\la,\la)B_{\la,\la}^{-1}e^{(h|\la)}U_{-\la}e^{-h}.
\end{equation*}
To prove the associativity in $\hat G$, note that the map $(\la,\mu) \mapsto B_{\la,\mu}$ is bimultiplicative by \eqref{Blambda}.
Since $\ep$ is a 2-cocycle, the map $(\la,\mu)\mapsto \ep(\la,\mu)B_{\la,\mu}^{-1}$ is also a 2-cocycle, 
which implies the associativity of $U_\la$, $U_\mu$ and $U_\nu$. The general case is then straightforward using \eqref{multinG}.
\end{proof}

%We observe that $\hat G$ is a central extension by $\CC^\times$ of the abelian direct product group $Q \times \exp(\lieh_0)$.
Recall the $\ph$-twisted Heisenberg algebra $\hhp$ from \deref{def:twaff}.
We introduce an action of $\hat G$ on $\hhp$ by conjugation:
\begin{equation}\label{Gcong}
\begin{split}
(c U_\la e^{h}) & (h'_{(m+\N)}+c'K)(c U_\la e^{h})^{-1} \\
&=h'_{(m+\N)}+\de_{m,0}\bigl{(} h'\big{|}\N h - \la\bigr{)}K + c'K.
\end{split}
\end{equation}
This action is compatible with \eqref{aUcomm} and the adjoint action of $\lieh_0$ on $\hhp$.
We say that an $\hhp$-module is an $(\hhp,\hat G)$-module if it is also a $\hat G$-module on which \eqref{Gcong} holds.
Such a module will be called \emph{restricted} if it is restricted as an $\hhp$-module (see \seref{ssbosons}).
%A $\hhp$-module or a $(\hhp,\hat G)$-module is called \emph{restricted} if for every $v \in M$ and $h \in \lieh$, we have $h_{(m+\N)}v = 0$ for $\re m \gg 0$. 
We say that an $(\hhp,\hat G)$-module is of \emph{level $1$} if both $K \in \hhp$ and $\mathbbm{1} \in \hat G$ act as the identity operator.

\subsection{The Subgroups $G$ and $N_\ph$}

We note that, as defined, the elements of the group $\hat G$ do not necessarily satisfy the condition \eqref{phieq} induced by $\ph$-equivariance. We will amend this issue by considering a quotient group in which $\eqref{phieq}$ holds. Consider
the elements
\begin{equation}\label{glambda}
g_\la = \eta(\la) e^{2\pi\ii c_\la}U_{\ph\la}^{-1}U_\la\tau_{\la} \in \hat G \qquad (\la \in Q),
\end{equation}
where we use the notation from \seref{subtvo}. 

In particular, recall that
\begin{equation}\label{tauh}
\tau_h = \th_h(2\pi\ii) = \exp \Bigl(\frac{1-e^{-2\pi\ii\N}}\N h\Bigr)_{(0+\N)} = \exp\Bigl(\frac{1-\ph}\N h\Bigr)_{(0+\N)}
\end{equation}
for $h\in\lieh$ (see \eqref{phsplit}, \eqref{modeproj}, \eqref{thetah}).
Notice that $\tau_h=\tau_{h_0}$ only depends on $h_0=\pi_0 h$.
By \eqref{thmult} and \eqref{thU}, we have:
\begin{equation}\label{multinG2}
\begin{split}
\tau_{h} \tau_{h'} &= \exp\Bigl(\frac{\ph+\ph^{-1}-2}{2\N} h_0 \Big| h' \Bigr) \tau_{h+ h'},\\
\tau_{h} U_\la \tau_{h}^{-1} &= \exp\Bigl(\frac{1-\ph}{\N} h_0 \Big| \la\Bigr) U_\la, \\
U_\la U_\mu &=\ep(\la,\mu)B_{\la,\mu}^{-1}U_{\la+\mu},
\end{split}
\end{equation}
and $\tau_{h}^{-1} = \tau_{-h}$. 

\begin{definition}\label{defG2}
Let $G$ be the subgroup of $\hat G$ consisting of all elements of the form
$c \,U_\la \tau_\mu$ ($c \in \CC^\times$, $\la,\mu \in Q$).
%\begin{equation*}%\label{groupel2}
%c \,U_\la \tau_\mu \qquad (c \in \CC^\times, \ \la,\mu \in Q).
%\end{equation*}
\end{definition}

We observe that $G$ contains as normal subgroups central extensions by $\CC^\times$ of the abelian additive groups $Q$ and $\pi_0 Q$.
%Moreover, both of these subgroups are normal in $G$.
%Note that $G$ is a central extension by $\CC^\times$ of the abelian direct product group $Q \times\pi_0 Q$.

\begin{remark}\label{heisgrp}
The subgroup $\CC^\times \times \exp(\lieh_0)$ of $\hat G$ is a Lie group with a Lie algebra $\hat\lieh_\ph^0$,
which is isomorphic to $\lieh_0\oplus\CC K$ with the Lie bracket $[h,h']=(\N h|h') K$ (see \eqref{hbrext}). This Lie algebra is a direct sum of abelian and finite-dimensional Heisenberg algebras. Similarly, the subgroup of $G$ generated by $\tau_\mu$ ($\mu\in Q$) is a direct product of an abelian group and a discrete Heisenberg group.
\end{remark}

\begin{proposition}\label{pnph}
The set\/ $N_\ph = \{ g_\la \,|\, \la \in Q \}$ is a central subgroup of\/ $G$. 
The action \eqref{Gcong} of\/ $N_\ph$ on\/ $\hhp$ is trivial.
\end{proposition}
\begin{proof}
We will first prove that the elements $g_\la$ are central in $\hat G$.
Since $g_\la=c U_{(1-\ph)\la}\tau_{\la}$ for some $c\in\CC^\times$, it will be enough to show that
$U_{(1-\ph)\la}\tau_{\la}$ is in the center of $\hat G$.
Then
\begin{equation*}
\begin{split}
e^{h}U_{(1-\ph)\la}&= e^{(h|(1-\ph)\la)}U_{(1-\ph)\la} e^{ h},\\
e^{h}\tau_{\la} & = e^{-(\N h | \frac{1}{\N}(\ph-1)\la)}\tau_{\la} e^{h} = e^{-(h|(1-\ph)\la)}\tau_{\la}  e^{h}
\end{split}
\end{equation*}
imply that $e^{h} U_{(1-\ph)\la}\tau_{\la} e^{-h}=U_{(1-\ph)\la}\tau_{\la}$ for $h \in \lieh_0$.
The same calculation shows that
\begin{equation*}
(U_{(1-\ph)\la}\tau_{\la}) h_{(m+\N)} (U_{(1-\ph)\la}\tau_{\la})^{-1} = h_{(m+\N)}.
\end{equation*}
Hence, $g_\la$ acts trivially on $\hhp$.

Next, we will show that $U_\mu U_{(1-\ph)\la}\tau_{\la} U_\mu^{-1}=U_{(1-\ph)\la}\tau_{\la}$ for $\la,\mu\in Q$. 
Using \eqref{phinv} and \eqref{lasum}, we compute
\begin{equation*}%\label{Ccomp1}
\begin{split}
e^{2\pi\ii(\frac{1-\si}{1-\ph}\mu_*|(1-\ph)\la)}&=e^{2\pi\ii((1-\si)\frac{1-\ph^{-1}}{1-\ph}\mu_*|\la)}\\
&=e^{-2\pi\ii((1-\si)\mu_*|\ph\la)}
=e^{2\pi\ii(\mu_0|\ph\la)}.
\end{split}
\end{equation*}
By a similar albeit longer computation, we obtain
\begin{equation*}%\label{Ccomp2}
e^{(\frac{1}{\N}(\pi\ii\N\frac{1+\ph}{1-\ph}-1)\mu_0|(1-\ph)\la)} = e^{-\pi \ii (\mu_0|(1+\ph)\la)}e^{2\pi\ii(\P^-_{2\pi\ii}\la_0|\mu)}.
\end{equation*}
Then using \eqref{cla}, we find
\begin{equation*}
C_{\mu,(1-\ph)\la} =e^{2\pi\ii(\P^-_{2\pi\ii}\la_0|\mu)}
\end{equation*}
and
%e^{\pi\ii(\mu_0|(1-\ph)\la)}e^{2\pi\ii(\frac{1-\ph}{1+\ph}\mu_*|(1-\ph)\la)}e^{(\frac{1}{\N}(\pi\ii\N\frac{1+\ph}{1-\ph})\mu_0|(1-\ph)\la)}
%\label{central2}
\begin{equation*}
U_\mu U_{(1-\ph)\la}\tau_{\la}U_\mu^{-1}=C_{\mu,(1-\ph)\la}e^{-2\pi\ii(\P_{2\pi\ii}^-\la_0|\mu)}U_{(1-\ph)\la}\tau_{\la} = U_{(1-\ph)\la}\tau_{\la}.
\end{equation*}
Therefore, $g_\la$ is central in $\hat G$.

Finally, to prove that $N_\ph$ is a subgroup of $G$, we will show that $g_\la g_\mu = g_{\la+\mu}$.
We compute
\begin{equation*}%\label{group1}
\begin{split}
e^{2\pi\ii (c_{\la}+c_\mu)} &= e^{-\pi\ii((\P^+_{2\pi\ii}+\P^-_{2\pi\ii})\la_0|\mu)}e^{2\pi\ii c_{\la+\mu}}, \\
%\label{group2}
\tau_{\la}U_\mu\tau_{\mu} &= e^{\pi\ii((\P^+_{2\pi\ii}+\P^-_{2\pi\ii})\la_0|\mu)}U_\mu\tau_{\la+\mu}, \\
%\label{group3}
(U_{\ph\la}U_{\ph\mu})^{-1}U_\la U_\mu &= \frac{\eta(\la+\mu)}{\eta(\la)\eta(\mu)}U_{\ph(\la+\mu)}^{-1}U_{\la+\mu}.
\end{split}
\end{equation*}
From these and the fact that $g_\la$ is central in $\hat G$, we obtain
 \begin{equation*}
 \begin{split}
 g_\la g_\mu& = \eta(\mu)e^{2\pi \ii c_\mu}U_{\ph\mu}^{-1}g_\la U_{\mu}\tau_{\mu}\\
 &=\eta(\la)\eta(\mu)e^{2\pi\ii(c_\la+c_\mu)}U_{\ph\mu}^{-1}U_{\ph\la}^{-1}U_{\la}\tau_{\la}U_\mu\tau_{\mu}\\
 &=\eta(\la+\mu) e^{2\pi\ii c_{\la+\mu}}U_{\ph(\la+\mu)}^{-1}U_{\la+\mu}\tau_{\la+\mu}\\
 &= g_{\la+\mu},
\end{split}
 \end{equation*}
which completes the proof.
\end{proof}

\subsection{Main Theorem}

 Consider the quotient group $G_\ph = G / N_\ph$. 
 Then having a $G_\ph$-module is equivalent to having a $G$-module on which the $\ph$-equivariance \eqref{phieq} holds. 
 By \prref{pnph}, the action \eqref{Gcong} of $G$ on $\hhp$ induces an action of $G_\ph$, and we can define the notion
 of an $(\hhp,G_\ph)$-module as in \seref{subtgg}.

The following theorem, which is the main result of the paper, reduces the classification of $\ph$-twisted $V_Q$-modules to the classification of restricted $(\hhp,G_\ph)$-modules.

\begin{theorem}\label{main}
Every\/ $\ph$-twisted\/ $V_Q$-module is naturally a restricted\/ $(\hhp, G_\ph)$-module of level\/ $1$. Conversely, any restricted\/ $(\hhp, G_\ph)$-module of level\/ $1$ extends to a\/ $\ph$-twisted\/ $V_Q$-module.
\end{theorem}

\begin{proof}
Let $M$ be a $\ph$-twisted $V_Q$-module.  
Then $Y(a,z)$ $(a \in \lieh)$ generate on $M$ the structure of a $\ph$-twisted module over the Heisenberg subalgebra $B^1(\lieh)\subset V_Q$.
Hence, $M$ is a restricted $\hhp$-module of level $1$ by \cite[Theorem 6.3]{Bak}.
We define the action of $U_\la$ $(\la \in Q)$ by \eqref{ulambda}, \eqref{Uops}.
The $\ph$-equivariance \eqref{phieq} is equivalent to $g_\la=\Id$ on $M$ for all $\la\in Q$ (cf.\ \eqref{glambda}).
This allows us to define $\tau_\la$ by 
\begin{equation}\label{taulambda}
\tau_\la = \eta(\la)^{-1} e^{-2\pi\ii c_\la} U_\la^{-1} U_{\ph\la}.
\end{equation}
Then the group relations of $G$ and the adjoint action \eqref{Gcong}  are satisfied due to \eqref{aUcomm}, \eqref{phieq}, and \eqref{Uprod}.  Thus $M$ is a restricted $(\hhp, G_\ph)$-module of level 1.

Conversely, assume $M$ is a restricted $(\hhp, G_\ph)$-module of level $1$. We define the logarithmic fields 
\begin{equation}\label{groupfields}
Y(a,z),\quad Y(e^\la,z), \qquad a \in \lieh, \ \la \in Q,
\end{equation}
by \eqref{tfield} and \eqref{twistedvop} respectively. Since the fields $Y(a,z)$ generate a $\ph$-twisted $B^1(\lieh)$-module structure on $M$ (see \cite[Theorem 6.3]{Bak}), 
they are local and satisfy the $\ph$-equivariance \eqref{phequiv}. 
The $\ph$-equivariance of $Y(e^\la,z)$ follows from \eqref{phieq}.
To show that $Y(a,z)$ is local with $Y(e^\la,z)$, we use \eqref{hvobrext}, \eqref{lafield} and compute
\begin{align*}
[Y(a,z_1),Y(e^\la,z_2)] 
&= \sum_{\al \in \A}\sum_{m \in \al} \bigl{(}z_1^{-m-1-\N} z_2^{m+\N} \pi_{\al}a \big{|}\la \bigr{)}Y(e^\la,z_2) \\
&= \de(z_1,z_2) \Bigl{(} \Bigl{(}\frac{z_2}{z_1}\Bigr{)}^{\mathcal{S}+\N}a\Big{|}\la\Bigr{)}Y(e^\la,z_2).
\end{align*}
Thus locality follows from the fact that $(z_1-z_2)\de(z_1,z_2)=0$ (see \eqref{delta}).
We will prove the locality of $Y(e^\la,z)$  and $Y(e^\mu,z)$ later.

Assuming that the fields $Y(e^\la,z)$, $Y(e^\mu,z)$ are local, consider the subspace $\mathcal{W} \subset \LF(M)$ spanned by the fields \eqref{groupfields}. The smallest subspace $\widehat{\mathcal{W}} \subset \LF(M)$ containing $\mathcal{W}\cup\{\Id\}$ and closed under $D_\ze$ and all $n$-th products is a local collection, and the $n$-th products endow $\widehat{\mathcal{W}}$ with the structure of a vertex algebra with vacuum vector $\Id$ and translation operator $D_z$ (see \cite[Theorem 3.7]{Bak}). A $\ph$-twisted $V_Q$-module structure on $M$ is equivalent to a homomorphism of vertex algebras $V_Q\to \widehat{\mathcal{W}}$
satisfying the $\ph$-equivariance \eqref{phequiv}. 

It is well known that the lattice vertex algebra $V_Q$ can be generated by the elements $h=h_{(-1)}\vac$ and $e^\la$ ($h\in\lieh$, $\la \in Q$), subject to the following relations:
\begin{enumerate}
\item\label{r1} $h_{(n)}h' = \de_{n,1}(h|h')\vac,$ \;\, $n \geq 0$,
\item\label{r2} $h_{(n)}e^\la = \de_{n,0}(h|\la)e^\la$, \; $n \geq 0$,
\item\label{r3} $Te^\la = \la_{(-1)}e^\la$,
\item\label{r4} ${e^\la}_{(-(\la|\mu)-1)}e^\mu = \ep(\la,\mu)e^{\la+\mu}$.
\end{enumerate}
Thus to show that $M$ carries the structure of a $\ph$-twisted $V_Q$-module, it suffices to check that the same relations are satisfied by the twisted fields \eqref{groupfields}. 
Relations \eqref{r1} and \eqref{r2} for the twisted fields are equivalent to \eqref{hbrext} and \eqref{hvobrext}; \eqref{r3} is equivalent to \eqref{Dm1}; and \eqref{r4} is equivalent to \eqref{VOPprod2}. All of these equations hold for the twisted fields by construction. 
%Hence, $M$ is a $\ph$-twisted $V_Q$-module.

It only remains to prove the locality of $Y(e^\la,z)$ and $Y(e^\mu,z)$. We do so by proving that 
\begin{equation}\label{Lvoploc}
\begin{split}
(z_1&-z_2)^{-(\la|\mu)}Y(e^\la,z_1)Y(e^\mu,z_2) \\
&= (-1)^{|\la|^2|\mu|^2+(\la|\mu)}(z_2-z_1)^{-(\la|\mu)}Y(e^\mu,z_2)Y(e^\la,z_1),
\end{split}
\end{equation} 
where if $(\la|\mu)>0$, we expand $(z_1-z_2)^{-(\la|\mu)}$ in the domain $|z_1|>|z_2|$, and $(z_2-z_1)^{-(\la|\mu)}$ in the domain $|z_2|>|z_1|$.
Let 
\begin{equation*}
R_{\la,\mu}(z_1,z_2) = (z_1-z_2)^{-(\la|\mu)} \exp(C), 
\end{equation*}
where $C$ is given by \eqref{commsum}. 
Then by \eqref{elzemz}, the left-hand side of \eqref{Lvoploc} is equal to
\begin{equation*}
R_{\la,\mu}(z_1,z_2) U_\la \th_\la(\ze_1)U_\mu\th_\mu(\ze_2) A_{\la,\mu}(z_1,z_2),
\end{equation*}
with
\begin{equation*}
\begin{split}
A_{\la,\mu}(z_1,z_2) &= e^{\ze_1 a_\la} e^{\ze_2 a_\mu} z_1^{b_\la} z_2^{b_\mu} E_\la(z_1)_+E_\mu(z_2)_+E_\la(z_1)_-E_\mu(z_2)_- \\
&=A_{\mu,\la}(z_2,z_1).
\end{split}
\end{equation*}
Using \eqref{thU} and \eqref{Uprod}, we find
\begin{equation*}
U_\la \th_\la(\ze_1)U_\mu\th_\mu(\ze_2) = C_{\la,\mu}
\exp\Bigl(-\frac{1}{\N}\Bigl(\Bigl(\frac{z_2}{z_1}\Bigr)^\N-1\Bigr)\la_0\Big|\mu\Bigr) U_\mu\th_\mu(\ze_2)U_\la\th_\la(\ze_1).
\end{equation*}
Thus to prove \eqref{Lvoploc}, it suffices to show that $R_{\la,\mu}(z_1,z_2)$ and
\begin{equation}\label{localeq}
(-1)^{|\la|^2|\mu|^2+(\la|\mu)}C_{\la,\mu}^{-1}
\exp\Bigl(\frac{1}{\N}\Bigl(\Bigl(\frac{z_2}{z_1}\Bigr)^\N-1\Bigr)\la_0\Big|\mu\Bigr) R_{\mu,\la}(z_2,z_1)
\end{equation}
are expansions of the same holomorphic function in the domains $|z_1|>|z_2|$ and $|z_2|>|z_1|$, respectively.

From \eqref{lerchder} and the proof of \prref{ecomb}, we have
\begin{equation}\label{Rfunct}
\exp(C) = \exp\Bigl{(}- \Bigl{(}\frac{z_2}{z_1}\Bigr{)}^{\mathcal{S}'+\N} \Phi\Bigl{(}\frac{z_2}{z_1},1,\mathcal{S}'+\N\Bigr{)}\la\Big{|}\mu \Bigr{)},
\end{equation}
where the right-hand side of \eqref{Rfunct} is expanded in the domain $|z_1|>|z_2|$. 
In the same domain, we have
\begin{equation*}
(z_1-z_2)^{-(\la|\mu)} = z_1^{-(\la|\mu)}\exp\Bigl{(}-(\la|\mu)\ln\Bigl{(}1-\frac{z_2}{z_1}\Bigr{)}\Bigr{)}.
\end{equation*}
Thus $R_{\la,\mu}(z_1,z_2)$ is the expansion of
\begin{equation}\label{localfunction}
 z_1^{-(\la|\mu)}\exp\Bigl{(}- \Bigl{(}\ln\Bigl{(}1-\frac{z_2}{z_1}\Bigr{)}+\Bigl{(}\frac{z_2}{z_1}\Bigr{)}^{\mathcal{S}'+\N} \Phi\Bigl{(}\frac{z_2}{z_1},1,\mathcal{S}'+\N\Bigr{)}\Bigr{)}\la\Big{|}\mu \Bigr{)}
\end{equation}
for $|z_1|>|z_2|$.
We note that \eqref{localfunction} is regular at $z_1=z_2$. Indeed, by \prref{ecomb}, 
\begin{equation*}
\begin{split}
- \Bigl(\Bigl{(}\ln\Bigl{(}1
&-\frac{z_2}{z_1}\Bigr{)}+\Bigl{(}\frac{z_2}{z_1}\Bigr{)}^{\mathcal{S}'+\N} \Phi\Bigl{(}\frac{z_2}{z_1},1,\mathcal{S}'+\N\Bigr{)}\Bigr{)}\la\Big{|}\mu \Bigr{)}\Big{|}_{z_1=z_2=z}\\
&=\bigl{(}(\Psi(\mathcal{S}'+\N)+\ga)\la\big{|}\mu\bigr{)}.
\end{split}
\end{equation*}

It follows that the function \eqref{localfunction} is holomorphic for $z_2/z_1 \notin \RR$. To complete the proof of locality, it suffices to show that \eqref{localeq} is equal to \eqref{localfunction} expanded in the domain $|z_2|>|z_1|$, $z_2/z_1 \notin \RR$.
Observe that as functions
\begin{equation*}%\label{lnref}
z_1^{-(\la|\mu)}\exp\Bigl{(}-\ln\Bigl{(}1-\frac{z_2}{z_1}\Bigr{)}\Bigr{)} = (-1)^{(\la|\mu)}z_2^{-(\la|\mu)}\exp\Bigl{(}-\ln\Bigl{(}1-\frac{z_1}{z_2}\Bigr{)}\Bigr{)}
\end{equation*}
for $z_2/z_1 \notin \RR$. Next, we write $\la = \la_0 + (1-\si)\la_*$ as in \eqref{lasum}, and consider separately the two resulting factors in \eqref{localfunction}.

For the first one, we use \eqref{al0pr} and the expansion formula \eqref{lerchref} to obtain
\begin{equation*}%\label{otherexp1}
\begin{split}
\Bigl{(} &\Bigl{(}\frac{z_2}{z_1}\Bigr{)}^{\mathcal{S}'+\N} \Phi\Bigl{(}\frac{z_2}{z_1},1,\mathcal{S}'+\N\Bigr{)}(1-\si)\la_*\Big{|}\mu \Bigr{)}\\
&=\pi\ii\sgn(\omega(z_2/z_1)) \bigl((1-\si)\la_*\big|\mu\bigr)
+\bigl{(}\pi\cot(\pi(\mathcal{S}'+\N))(1-\si)\la_*\big{|} \mu \bigr{)}\\
&\quad+\Bigl{(}\Bigl{(}\frac{z_1}{z_2}\Bigr{)}^{\mathcal{S}'+\N} \Phi\Bigl{(}\frac{z_1}{z_2},1,\mathcal{S}'+\N\Bigr{)}(1-\si)\mu_*\Big{|}\la\Bigr{)}.
\end{split}
\end{equation*}
To simplify the second factor resulting from \eqref{localfunction},
notice that \eqref{lerchref} and \eqref{lerchshift} imply
\begin{align*}
z^{1+a} &\Phi(z,1,1+a) = z^{a}\Phi(z,1,a) - \frac{z^a}{a} \\
&= \pi\ii\sgn\omega(z) + \pi\cot(\pi a) + z^{a-1}\Phi(z^{-1},1,1-a) - \frac{z^a}{a} \,.
\end{align*}
Then
\allowdisplaybreaks
\begin{align*}%\label{otherexp2}
\Bigl{(}\Bigl{(}&\frac{z_2}{z_1}\Bigr{)}^{1+\N}\Phi\Bigl{(}\frac{z_2}{z_1},1,1+\N\Bigr{)}\la_0\Big{|}\mu\Bigr{)}\\
&=\pi\ii \sgn(\omega(z_2/z_1)) \, (\la_0|\mu)
-\Bigl(\frac{1}{\N} \Bigl(\Bigl(\frac{z_2}{z_1}\Bigr)^\N-1\Bigr)\la_0\Big|\mu\Bigr) \\
&\quad+\Bigl{(}\frac{1}{\N}\bigl(\pi\N\cot(\pi\N)-1\bigr)\la_0\Big{|}\mu\Bigr{)}\\
&\quad +\Bigl{(}\Bigl{(}\frac{z_1}{z_2}\Bigr{)}^{1-\N}\Phi\Bigl{(}\frac{z_1}{z_2},1,1-\N\Bigr{)}\la_0\Big{|}\mu\Bigr{)}\\
&=\pi\ii \sgn(\omega(z_2/z_1)) \, (\la_0|\mu)
+\Bigl(\frac{1}{\N} \Bigl(\Bigl(\frac{z_1}{z_2}\Bigr)^\N-1\Bigr)\mu_0\Big|\la\Bigr) \\
&\quad+\Bigl{(}\frac{1}{\N}\bigl(\pi\N\cot(\pi\N)-1\bigr)\la_0\Big{|}\mu\Bigr{)}\\
&\quad +\Bigl{(}\Bigl{(}\frac{z_1}{z_2}\Bigr{)}^{1+\N}\Phi\Bigl{(}\frac{z_1}{z_2},1,1+\N\Bigr{)}\mu_0\Big{|}\la\Bigr{)}.
\end{align*}
Plugging these into \eqref{localfunction} and using the expression for $C_{\la,\mu}$ from the proof of \thref{commutant},
we obtain \eqref{localeq}.
\end{proof}

%%%%%%%%%%%%%%%%%%%%%%%%%%%%%%%%%%%%%%%%%%%%%%%%%
%%%%%%%%%%%%%%%%%%%%%%%%%%%%%%%%%%%%%%%%%%%%%%%%%
\section{Examples of $\ph$-twisted $V_Q$-modules}\label{septvm}
%%%%%%%%%%%%%%%%%%%%%%%%%%%%%%%%%%%%%%%%%%%%%%%%%
%%%%%%%%%%%%%%%%%%%%%%%%%%%%%%%%%%%%%%%%%%%%%%%%%

%A major advantage of the theory of twisted logarithmic modules of vertex algebras as developed in \cite{Bak} is the framework for working out explicit examples. 
In this section, we demonstrate the utility of the results of the previous section by working out two small-dimensional examples in detail.

\subsection{Example of a Rank $4$ Lattice}
We begin with the integral lattice
\begin{equation*}
Q = \Span_\ZZ\{\la_1,\la_2,\la_3,\la_4\},
\end{equation*}
with bilinear form $(\cdot | \cdot)\colon Q \times Q \to \ZZ$ defined by
\begin{equation*}
(\la_1|\la_4)=(\la_4|\la_1)=(\la_2|\la_3)=(\la_3|\la_2)=1,
\end{equation*}
and all other scalar products of basis elements are zero. We let $\ph\colon Q \to Q$ be the automorphism given by 
\begin{equation*}
\ph \la_1 = \la_1-\la_2, \qquad \ph \la_2 = \la_2, \qquad \ph\la_3 = \la_3+\la_4, \qquad \ph\la_4 = \la_4.
\end{equation*}
It is easy to check that $(\cdot|\cdot)$ is $\ph$-invariant. We let $\lieh = \CC \otimes_\ZZ Q $. Then we can write $\ph = e^{-2\pi\ii \N}$, where 
\begin{equation*}
\N\la_1 = \frac{1}{2\pi\ii}\la_2,\qquad \N \la_2 = 0, \qquad \N \la_3 = -\frac{1}{2\pi\ii}\la_4,\qquad \N\la_4=0,
\end{equation*}
and $\N^2\la_1 = \N^2 \la_3 = 0$. We use the $2$-cocycle
\begin{equation*}%\label{4dimcoc}
\ep(\la_4,\la_1)=\ep(\la_3,\la_2)=-1,
\end{equation*}
and $\ep=1$ on all other pairs of basis vectors. This allows us to choose $\eta = 1$ in \eqref{etaep}, 
since $\ep(\la_i,\la_j)=\ep(\ph\la_i,\ph\la_j)$ for all $i,j$.

Recall that the $\ph$-twisted Heisenberg algebra $\hat{\h}_\ph$ has a triangular decomposition 
given by \eqref{wtriangle}.
We consider the $\hhp^0$-module
\begin{equation*} R=\CC[x_{1,0},q_1^{\pm 1},q_2^{\pm 1}],
\end{equation*}
with the action
\begin{align*}
\la_{1(0+\N)} &= x_{1,0}, & \la_{2(0+\N)} &= q_1\partial_{q_1}, \\
\la_{3(0+\N)} &= -\frac{1}{2\pi\ii}\partial_{x_{1,0}}, & \la_{4(0+\N)}&=q_2\partial_{q_2},
\end{align*}
and we let $M_\ph(R)$ be the corresponding generalized Verma module for $\hhp$. We label the action of $\hhp^-$ by the commuting variables:
\begin{align*}
\la_{1(-m+\N)} &= x_{1,m}, &\la_{2(-m+\N)}&=2\pi\ii \, x_{2,m},\\
\la_{3(-m+\N)}&=\frac{1}{2\pi\ii}x_{3,m}, & \la_{4(-m+\N)}&=x_{4,m},
\end{align*}
for $m\in \NN$. Then
\begin{equation*}
M_\ph(R) \cong \CC[x_{1,0},q_1^{\pm 1},q_2^{\pm 1},x_{i,m}]_{1\leq i \leq 4;m=1,2,\ldots},
\end{equation*}
and the action of $\hhp^+$ is given by
\begin{align*}
\la_{1(m+\N)} &= m\partial_{x_{4,m}}+\partial_{x_{3,m}},&
\la_{2(m+\N)}&=2\pi\ii \, m\partial_{x_{3,m}}, \\
\la_{3(m+\N)} &= \frac{1}{2\pi\ii}\bigl{(}m \partial_{x_{2,m}}-\partial_{x_{1,m}}\bigr{)},&
\la_{4(m+\N)} & = m \partial_{x_{1,m}},
\end{align*}
for $m\in \NN$.

We define the following operators on $M=M_\ph(R)[e^{\pm 2\pi\ii x_{1,0}}]$:
\begin{align*}%\label{4dimU}
U_{\la_1} &= q_2, &U_{\la_2}& = (-1)^{q_1\partial_{q_1}}e^{-2\pi\ii x_{1,0}},\\
 U_{\la_3} &= q_1e^{-\frac{\pi\ii}{6}q_2\partial_{q_2}}, &U_{\la_4} &= (-1)^{q_2\partial_{q_2}}e^{-\partial_{x_{1,0}}},
\end{align*}
and $\tau_\la$ by \eqref{taulambda}.
It is straightforward to check that \eqref{aUcomm} holds, and the operators $U_\la$ satisfy the following commutation relations:
\begin{align*}
U_{\la_1}U_{\la_2}&=U_{\la_2}U_{\la_1},  &U_{\la_2}U_{\la_3}&=-U_{\la_3}U_{\la_2},\\
U_{\la_1}U_{\la_3}&=e^{\frac{\pi\ii}{6}}U_{\la_3}U_{\la_1}, & U_{\la_2}U_{\la_4}&= U_{\la_4}U_{\la_2},\\
U_{\la_1}U_{\la_4}&=-U_{\la_4}U_{\la_1}, & U_{\la_3}U_{\la_4}&=U_{\la_4}U_{\la_3},
\end{align*}
which agree with the relations \eqref{Uprod}.
Hence $M$ is a representation of the group $G$ from \deref{defG2}. It is also straightforward to check that the compatibility \eqref{Gcong} holds and all elements of $N_\ph$ act as the identity on $M$. Thus $M$ is a restricted $(\hhp,G_\ph)$-module of level 1.  By \thref{main}, the logarithmic fields 
\begin{equation*}%\label{4dimfields}
Y(a,z), \qquad a \in \lieh,
\end{equation*}
and 
\begin{equation*}%\label{4dimvert}
Y(e^{\la},z) = U_{\la}\th_{\la}e^{\ze a_{\la}}z^{b_{\la}}E_{\la}(z),\qquad \la \in Q,
\end{equation*}
generate a $\ph$-twisted $V_Q$-module structure on $M$. 

Explicitly, the logarithmic vertex operators corresponding to the generators $e^{\la_i}$ are:
\allowdisplaybreaks
\begin{align*}
Y(e^{\la_1},z) &= q_2 e^{\ze x_{1,0}}e^{-\frac{\ze^2}{4\pi\ii}q_1 \partial_{q_1}}
\exp\Bigl{(}\sum_{n =1}^\infty\Bigl{(}\frac{1}{n} x_{1,n}+\frac{1-n\ze}{n^2}x_{2,n}\Bigr{)}z^n\Bigr{)}\\
&\quad\times\exp\Bigl{(}-\sum_{n=1}^\infty \bigl{(}\partial_{x_{4,n}}-\ze\partial_{x_{3,n}} \bigr{)}z^{-n}\Bigr{)}, \\
 Y(e^{\la_2},z) &= (-1)^{q_1\partial_{q_1}}e^{-2\pi\ii x_{1,0}}z^{q_1 \partial_{q_1}}
 \exp \Bigl{(}2\pi\ii\sum_{n=1}^\infty x_{2,n}\frac{z^n}{n}\Bigr{)}\\
 &\quad\times\exp \Bigl{(} -2\pi\ii\sum_{n=1}^\infty \partial_{x_{3,n}}z^{-n}\Bigr{)}, \\
Y(e^{\la_3},z) & = q_1 e^{(\frac{\ze^2}{4\pi\ii}-\frac{\pi\ii}{6})q_2 \partial_{q_2}}e^{-\frac{\ze}{2\pi\ii}\partial_{x_{1,0}}}\\
&\quad\times\exp\Bigl{(} \frac{1}{2\pi\ii} \sum_{n=1}^\infty\Bigl{(} \frac{1}{n}x_{3,n}+\frac{n\ze-1}{n^2}x_{4,n}\Bigr{)}z^n\Bigr{)}\\
&\quad\times\exp\Bigl{(}-\frac{1}{2\pi\ii}\sum_{n=1}^\infty\Bigl{(}\partial_{x_{2,n}}+\ze \partial_{x_{1,n}} \Bigr{)}z^{-n} \Bigr{)}, \\
Y(e^{\la_4},z) & = (-1)^{q_2 \partial_{q_2}}e^{-\partial_{x_{1,0}}}z^{q_2 \partial_{q_2}}
\exp\Bigl{(}\sum_{n=1}^\infty x_{4,n}\frac{z^n}{n} \Bigr{)}
\exp \Bigl{(}-\sum_{n=1}^\infty \partial_{x_{1,n}}z^{-n} \Bigr{)}.
\end{align*}

The vectors
\begin{equation*}
v_1=\la_1, \qquad v_2 = \frac{\la_2}{2\pi\ii}, \qquad v_3 = 2\pi\ii \la_3, \qquad v_4 = \la_4
\end{equation*}
form a basis for $\lieh$ for which $\ph$ and $(\cdot|\cdot)$ are as in \exref{ex:symm1}. Let $\om\in \lieh$ be the conformal vector given by \eqref{fbomega}. Then the action of the Virasoro operator $L_0$ on $M$ is given by \prref{prop:bvir} with $\al_0=0$:
\begin{equation*}
\begin{split}
L_0 &= \sum_{i=1}^4\sum_{n=1}^\infty n x_{i,n}\partial_{x_{i,n}}+ \sum_{n=1}^\infty \bigl{(} x_{4,n}\partial_{x_{3,n}}-x_{2,n}\partial_{x_{1,n}}\bigr{)}\\
&\qquad +x_{1,0}q_2\partial_{q_2}+\frac{1}{2\pi\ii}q_1\partial_{q_1}\partial_{x_{1,0}}.
\end{split}
\end{equation*}

\subsection{Example of a Rank $3$ Lattice}

 For an example on which $\N$ acts as a single Jordan block, we revisit Example 6.10 from \cite{Bak}. Let $\lieh$ be the Cartan subalgebra of a type $A_1^{(1)}$ affine Kac--Moody algebra. The dual space $\lieh^*$ has a basis $\{\al_1,\de, \La_0\}$ with a nondegenerate symmetric bilinear form $(\cdot|\cdot)$ defined by:
\begin{equation*}
(\al_1|\al_1)=2, \qquad (\de|\La_0)=(\La_0|\de)=1,
\end{equation*}
and the other products of basis vectors equal to 0. Consider the integral lattice $Q = \Span_\ZZ\{\al_1,\de, \La_0\}$. We let $\ph=t_{\al_1}$ be the element of the affine Weyl group which acts on $\lieh^*$ by 
\begin{equation*}
\ph\al_1 = \al_1-2\de, \qquad  \ph\de = \de, \qquad \ph\La_0= \La_0+\al_1-\de.
\end{equation*}
Then $\ph$ is an automorphism of the lattice, and $(\cdot|\cdot)$ is $\ph$-invariant. 
We can write $\ph=e^{-2\pi\ii\N}$, where 
\begin{equation*}
\N\al_1 = \frac{\de}{\pi\ii}, \qquad \N\de = 0, \qquad \N\La_0 = -\frac{\al_1}{2\pi\ii}\,, \qquad \N^2 \La_0 = \frac{\de}{2\pi^2},
\end{equation*}
and $\N^2\al_1 = \N^3\La_0 = 0$. We use the $2$-cocycle $\ep$ with 
\begin{equation*}
\ep(\al_1,\al_1)=\ep(\de,\La_0) = -1,
\end{equation*}
and $\ep=1$ on all other pairs of generators. Again we can assume $\eta=1$ on $Q$.

We identify $\lieh$ with its dual space and write $\lieh=\CC \otimes_{\ZZ}Q$. We consider the $\hhp^0$-module
\begin{equation*} R=\CC[x_{1,0},q^{\pm 1}],
\end{equation*}
with the action
\begin{equation*}
\La_{0(0+\N)}=-\frac{\sqrt{2}}{2\pi\ii} \,x_{1,0}, \qquad \al_{1(0+\N)}=-\sqrt{2} \, \partial_{x_{1,0}}, \qquad \de_{(0+\N)} = q \partial_q.
\end{equation*}
We let $M_\ph(R)$ be the corresponding generalized Verma module for $\hhp$. We label the action of $\hhp^-$ by the commuting variables:
\begin{equation*}
\La_{0(-m+\N)} = -\frac{\sqrt{2}}{2\pi\ii} \, x_{1,m},\quad \al_{1(-m+\N)} = \sqrt{2}\, x_{2,m}, \quad \de_{(-m+\N)} = -\frac{2\pi\ii}{\sqrt{2}}\,x_{3,m}, 
\end{equation*}
for $m\in \NN$. Then
\begin{equation*}
M_\ph(R)\cong \CC[x_{1,0},q^{\pm 1},x_{i,m}]_{1\leq i\leq 3; m=1,2,\ldots},
\end{equation*}
and the action of $\hhp^+$ is given by
\begin{equation*}
\begin{split}
\La_{0(m+\N)}& = -\frac{\sqrt{2}}{2\pi\ii}\bigl{(}m \partial_{x_{3,m}}+\partial_{x_{2,m}}\bigr{)},\\
\al_{1(m+\N)}& = \sqrt{2} \bigl{(}m \partial_{x_{2,m}}-\partial_{x_{1,m}}\bigr{)},\\
\de_{(m+\N)} & = -\frac{2\pi\ii}{\sqrt{2}} \,m\partial_{x_{1,m}}.
\end{split}
\end{equation*}

Let $M=M_\ph(R)[e^{\pm\sqrt{2}\, x_{1,0}}]$.
We define the operators $U_\la$ on $M$ by
\begin{equation*}
%U_{\La_0}=q, \qquad U_{\al_0}=-\ii e^{2\pi\ii \La_{(0+\N)}} e^{\frac{\pi\ii}{3}\de_{(0+\N)}},\qquad U_\de=e^{-\pi\ii \al_{(0+\N)}}e^{\pi\ii \de_{(0+\N)}}
U_{\La_0}=q, \qquad U_{\al_1}=-\ii e^{\frac{\pi\ii}{3}q\partial_q} e^{-\sqrt{2}\, x_{1,0}},\qquad U_\de=(-1)^{q\partial_q} e^{\frac{2\pi\ii}{\sqrt{2}}\partial_{x_{1,0}}},
\end{equation*}
and $\tau_\la$ by \eqref{taulambda}.
These operators agree with \eqref{Uprod}. In particular,
\begin{equation*}
\begin{split}
U_{\La_0}U_{\al_1}&=e^{-\frac{\pi \ii}{3}}U_{\al_1}U_{\La_0},\\
U_{\La_0}U_{\de}&=-U_{\de}U_{\La_0},\\
U_{\al_1}U_{\de}&=U_{\de}U_{\al_1}.
\end{split}
\end{equation*}
As in the previous example, $M$ is a restricted $(\hhp,G_\ph)$-module of level $1$, and hence can be extended to be a $\ph$-twisted $V_Q$-module. 

Explicitly, the logarithmic vertex operators corresponding to the generators are:
\begin{align*}
Y&(e^{\La_0},z)=q e^{-\frac{\sqrt{2}}{2\pi\ii}\ze x_{1,0}}e^{-\frac{\sqrt{2}}{4\pi\ii}\ze^2\partial_{x_{1,0}}+\frac{\ze^3}{12\pi^2}q\partial_q}e^{\frac{\ze^3}{24\pi^2}}\\
&\times \exp\Bigl{(} -\frac{\sqrt{2}}{2\pi\ii}\sum_{n=1}^\infty\Bigl{(}\frac{1}{n}x_{1,n}+\frac{1-n\ze}{n^2} x_{2,n}-\frac{n^2\ze^2-2n\ze+2}{2 n^3}x_{3,n}\Bigr{)}z^n\Bigr{)}\\
&\times \exp\Bigl{(}\frac{\sqrt{2}}{2\pi\ii}\sum_{n=1}^\infty\Bigl{(} \partial_{x_{3,n}}-\ze\partial_{x_{2,n}}-\frac{\ze^2}{2}\partial_{x_{1,n}}\Bigr{)}z^{-n} \Bigr{)}, \\
Y&(e^{\al_1},z)= -\ii e^{\frac{\pi\ii}{3}q\partial_q} e^{-\sqrt{2}\, x_{1,0}}e^{-\sqrt{2}\ze\partial_{x_{1,0}}-\frac{\ze^2}{2\pi\ii}q\partial_q}\\
&\qquad\times \exp\Bigl{(}\sqrt{2}\sum_{n=1}^\infty\Bigl{(}\frac{1}{n}x_{2,n}+\frac{n\ze-1}{n^2}x_{3,n}\Bigr{)}z^n\Bigr{)}\\
&\qquad\times\exp\Bigl{(}-\sqrt{2}\sum_{n=1}^\infty\Bigl{(}\partial_{x_{2,n}}+\ze \partial_{x_{1,n}}\Bigr{)}z^{-n}\Bigr{)}, \\
Y&(e^\de,z)=(-1)^{q\partial_q} e^{\frac{2\pi\ii}{\sqrt{2}}\partial_{x_{1,0}}}z^{q\partial_q}
\exp\Bigl{(}-\frac{2\pi \ii}{\sqrt{2}}\sum_{n=1}^\infty x_{3,n}\frac{z^n}{n}\Bigr{)}\\
&\qquad\times\exp\Bigl{(}\frac{2\pi\ii}{\sqrt{2}}\sum_{n=1}^\infty \partial_{x_{1,n}}z^{-n}\Bigr{)}.
\end{align*}
After a change of variables (cf.\ \cite[Remark 2.2]{BW}) and a certain reduction, the above operators $Y(e^{\al_1},z)$ and $Y(e^\de,z)$ are related to Milanov's vertex operators from \cite{M}.

The vectors
\begin{equation*}
v_1 = -\frac{2\pi\ii}{\sqrt{2}}\La_0, \qquad v_2 = \frac{\al_1}{\sqrt{2}}, \qquad v_3 = -\frac{\sqrt{2}}{2\pi\ii}\de
\end{equation*}
form a basis for $\lieh$ for which $\ph$ and $(\cdot|\cdot)$ are as in \exref{ex:symm2}. Again we let $\om\in \lieh$ be the conformal vector given by \eqref{fbomega}; then the action of the Virasoro operator $L_0$ on $M$ is given by \prref{prop:bvir} with $\al_0=0$:
\begin{equation*}
\begin{split}
L_0 & = \sum_{i=1}^3\sum_{n=1}^\infty n x_{i,n}\partial_{x_{i,n}}+\sum_{n=1}^\infty\bigl{(}x_{3,n}\partial_{x_{2,n}}-x_{2,n}\partial_{x_{1,n}} \bigr{)}\\
&\qquad -\frac{\sqrt{2}}{2\pi\ii}x_{1,0}q\partial_q+\frac{1}{2}\partial_{x_{1,0}}^2.
\end{split}
\end{equation*}

%%%%%%%%%%%%%%%%%%%%%%%%%%%%%%%%%%%%%%%%%%%%%%%%%
%%%%%%%%%%%%%%%%%%%%%%%%%%%%%%%%%%%%%%%%%%%%%%%%%
\section*{Acknowledgements}
%%%%%%%%%%%%%%%%%%%%%%%%%%%%%%%%%%%%%%%%%%%%%%%%%
%%%%%%%%%%%%%%%%%%%%%%%%%%%%%%%%%%%%%%%%%%%%%%%%%
The authors wish to thank the anonymous referee for carefully reading the paper and for their thoughtful comments and questions. The first author is supported in part by a Simons Foundation grant 279074.

%%
%%
%%
%%
%%
%%

%%%%%%%%%%%%%%%%%%%%%%%%%%%%%%%%%%%%%%%%%%%%%%%%%
%%%%%%%%%%%%%%%%%%%%%%%%%%%%%%%%%%%%%%%%%%%%%%%%%
%\section{Conclusion}
%%%%%%%%%%%%%%%%%%%%%%%%%%%%%%%%%%%%%%%%%%%%%%%%%
%%%%%%%%%%%%%%%%%%%%%%%%%%%%%%%%%%%%%%%%%%%%%%%%%

%%%%%%%%%%%%%%%%%%%%%%%%%%%%%%%%%%%
% \bibliography{freefields}{}
% \bibliographystyle{plain}
%%%%%%%%%%%%%%%%%%%%%%%%%%%%%%%%%%%
\bibliographystyle{amsalpha}

\end{document}